\documentclass[a4paper,10pt]{article}

\usepackage{amsmath,amsthm,amssymb,amsfonts,extarrows}
\usepackage{enumerate,amscd,amsxtra,MnSymbol}
\usepackage{mathrsfs}
\usepackage{color}
\usepackage[cmtip,all]{xy}
\usepackage{eucal}
\usepackage{bbold}
\usepackage{ upgreek }
\usepackage{paralist}

\theoremstyle{plein}

\newtheorem{theorem}{Theorem}[section]
\newtheorem*{theorem*}{Theorem}
\newtheorem{lemma}[theorem]{Lemma}
\newtheorem{proposition}[theorem]{Proposition}
\newtheorem*{proposition*}{Proposition}
\newtheorem{corollary}[theorem]{Corollary}
\newtheorem*{corollary*}{Corollary}

\newtheorem{conjecture*}{Conjecture}

\theoremstyle{definition}
\newtheorem{example}[theorem]{Example}
\newtheorem{definition}[theorem]{Definition}
\newtheorem{definition*}{Definition}
\newtheorem{remark}[theorem]{Remark}
\newtheorem{remark*}{Remark}
\newtheorem{notation}[theorem]{Notation}
\newtheorem{recollection}{Recollection}[section]

\newtheorem{construction}[theorem]{Construction}

\newcommand{\bA}{{\mathbb A}}
\newcommand{\bC}{{\mathbb C}}
\newcommand{\bE}{{\mathbb E}}

\newcommand{\bG}{{\mathbb G}}

\newcommand{\ZZ}{{\mathbb Z}}

\newcommand{\mA}{{\mathcal A}}
\newcommand{\mB}{{\mathcal B}}
\newcommand{\mC}{{\mathcal C}}
\newcommand{\mD}{{\mathcal D}}
\newcommand{\mE}{{\mathcal E}}

\newcommand{\mJ}{{\mathcal J}}
\newcommand{\mK}{{\mathcal K}}

\newcommand{\mS}{{\mathcal S}}

\newcommand{\mU}{{\mathcal U}}

\newcommand{\mX}{{\mathcal X}}

\newcommand{\A}{{\mathrm A}}
\newcommand{\B}{{\mathrm B}}
\newcommand{\C}{{\mathrm C}}

\newcommand{\E}{{\mathrm E}}
\newcommand{\F}{{\mathrm F}}
\newcommand{\G}{{\mathrm G}}
\newcommand{\rH}{{\mathrm H}}

\renewcommand{\L}{{\mathrm L}}
\newcommand{\M}{{\mathrm M}}

\newcommand{\R}{{\mathrm R}}
\newcommand{\rS}{{\mathrm S}}

\newcommand{\V}{{\mathrm V}}

\newcommand{\X}{{\mathrm X}}
\newcommand{\Y}{{\mathrm Y}}
\newcommand{\Z}{{\mathrm Z}}

\newcommand{\rc}{{\mathrm c}}

\newcommand{\bj}{{\mathrm j}}
\newcommand{\bi}{{\mathrm i}}
\newcommand{\m}{{\mathrm m}}
\newcommand{\bk}{{\mathrm k}}

\newcommand{\g}{{\mathrm g}}
\newcommand{\n}{{\mathrm n}}

\newcommand{\op}{\mathrm{op}}
\newcommand{\dual}{\vee}

\newcommand{\MGL}{\mathsf{MGL}}

\newcommand{\SH}{\mathsf{SH}}

\newcommand{\KGL}{\mathsf{KGL}}

\newcommand{\Sm}{\mathrm{Sm}}

\newcommand{\Sp}{\mathrm{Sp}}

\newcommand{\y}{\mathsf{y}}
\newcommand{\f}{\mathsf{f}}

\newcommand{\colim}{\mathrm{colim}}

\newcommand{\Mod}{{\mathrm{Mod}}}

\newcommand{\RMod}{{\mathrm{RMod}}}

\newcommand{\ot}{\otimes}

\newcommand{\Pre}{\mathrm{Pre}}

\renewcommand{\smash}{\wedge}

\newcommand{\id}{\mathrm{id}}

\newcommand{\Cat}{\mathsf{Cat}}

\newcommand{\Pic}{\mathrm{Pic}}

\newcommand{\Alg}{\mathrm{Alg}}

\newcommand{\Fun}{\mathrm{Fun}}
\newcommand{\tu}{{\mathbb 1}}
\newcommand{\Op}{{\mathrm{Op}}}

\newcommand{\lax}{{\mathrm{lax}}}

\newcommand{\Mul}{{\mathrm{Mul}}}
\newcommand{\lan}{{\mathrm{lan}}}
\newcommand{\ev}{{\mathrm{ev}}}
\newcommand{\Ind}{{\mathrm{Ind}}}

\newcommand{\Cell}{\mathrm{Cell}}

\newcommand{\Rig}{\mathrm{Rig}}

\begin{document}
	
\author{Hadrian Heine}
 
\title{A topological model for cellular motivic spectra} 

\maketitle 

\begin{abstract}

For any motivic $\bE_\infty$-ring spectrum $A$ we construct an equivalence
$\rho$ between the $\infty$-category of cellular motivic $A$-module spectra and modules over an $\bE_1$-algebra $\Theta$ in $\ZZ $-graded spectra,
under which the motivic grading corresponds to the $\ZZ$-grading.
If the base is $\bC$ or if $A$ admits an $\bE_\infty$-orientation,
we refine the $\bE_1$-algebra $\Theta$ to an $\bE_\infty$-algebra
and $\rho$ to a symmetric monoidal equivalence.

To capture the symmetric monoidal structure in the general situation, we lift $\rho$ to a symmetric monoidal equivalence to modules over an $\bE_\infty$-algebra in $\mJ $-graded spectra that invert morphisms of $\mJ$, where $\mJ$ is the diagram category 
of Sagave-Schlichtkrull \cite{Sagave2011DiagramSA}, a model for Quillen's localization of the groupoid of finite sets and bijections.

\end{abstract}

\tableofcontents

\section{Introduction}

Stable motivic homotopy theory has been at the heart of a recent breakthrough in
the computation of stable homotopy groups of spheres \cite{Isaksen2014StableS} \cite{unknown}.
Instead of identifying the Adams spectral sequence of the sphere, the main tool to compute stable homotopy groups, one identifies the $\bC$-motivic analogue of the
Adam's spectral sequence that realizes to the classical Adam's spectral sequence and can be usefully related \cite[Theorem 6.1]{Isaksen2014StableS} to the algebraic Novikov spectral sequence \cite{miller}, which is computable by machine.

As a consequence of these achievements the question arose if the computations of
stable homotopy groups \cite{Isaksen2014StableS}\cite{unknown},
which use deep results of motivic homotopy theory like Voevodsky's computations of
the motivic Steenrod algebra and motivic cohomology of a point, necessarily depend on motivic homotopy theory.

The work of \cite{mmf} gives an answer to this question: the authors prove that the computations of \cite{Isaksen2014StableS}\cite{unknown} are independent
from motivic homotopy theory by presenting the homotopy theory of 2-complete complex cellular motivic spectra by purely topological data:
they construct a lax symmetric monoidal functor $\Xi_*: \Sp \to \Sp^\mathrm{fil}$ from spectra to filtered spectra and model 2-complete complex cellular motivic spectra as $\Xi_*(S^0)$-modules in filtered spectra \cite[Theorem 6.12]{mmf}.
For every $n \in \ZZ$ the spectrum $\Xi_n(S^0)$ is identified with the spectrum of global sections of the motivic spectrum $\rS^{0,-n}$ \cite[Proposition 6.8]{mmf},
where $S^{p,q}:= S^{p} \smash (\mathbb{G}_m[-1])^{\smash q}$ for $p,q \in \ZZ$
is the bigraded motivic sphere.

The authors use this presentation of 2-complete complex cellular motivic spectra to construct motivic analogues of classical spectra of importance.

Applying the functor $\Xi_*$ to a spectrum $X$ gives a $\Xi_*(S^0)$-module $\Xi_*(X)$ modeling a 2-complete complex cellular motivic spectrum, which serves as a motivic analogue of $X$ according to the following table:

$$\begin{tabular}{ l | c}
X & $\Xi_*(X)$ \\
\hline $H\mathbb{F}_2$ & $H\mathbb{F}_2$\\
$H\mathbb{Z}$ & $H\mathbb{Z}$ \\
$KU$ & $KGL$\\
$KO$ & $KQ$\\
$BP$ & $BPGL$\\
$MU$ & $MGL$\\
$tmf$ & $mmf$
\end{tabular}$$

The intuition why this topological presentation is possible, lies in the following
observation: cellular motivic spectra are the largest class of motivic spectra $Z$
that are controlled by bigraded stable homotopy groups $ \pi_{p,q}(Z)$ for $p,q \in \ZZ$, defined as homotopy classes of maps from the motivic sphere $S^{p,q} $ to $Z.$
Consequently, to model cellular motivic spectra by topological data one needs to add an extra grading to the topological grading that records the grading of the twist $(-)(q) := (-) \smash (\mathbb{G}_m[-1])^{\smash q}$.
Said differently, the twist is transformed into the shift of $\ZZ$-graded objects
and $\pi_{p,q}(Z)= \pi_p(Z'_q)$ when $Z'$ is the topological model for $Z.$



In this article we extend the result of \cite{mmf} from 2-complete complex cellular motivic spectra to cellular motivic spectra over any base scheme 
and cellular modules over such.
To state our results we use the following $\infty$-categories:
for every motivic $\bE_\infty$-ring spectrum $A$ over some base scheme $\rS$ 
let $\Mod_A(\SH(\rS))^{\mathrm{cell}}$ be the symmetric monoidal $\infty$-category of cellular motivic $A$-modules. 
Assigning the free $A$-module on the $\mathbb{P}^1$-stabilization of the $\bA^1$-localization of the constant Nisnevich sheaf gives a symmetric monoidal functor $\Sp \to \Mod_A(\SH(\rS))^{\mathrm{cell}}$ that admits a right adjoint
$\Gamma$, wich is lax symmetric monoidal and takes the spectrum of global sections.
Let $\Sp^\ZZ$ be the $\infty$-category of $\ZZ$-graded spectra that carries a symmetric monoidal structure given by Day-convolution, for which $\bE_\bk$-algebras  
for $1 \leq \bk \leq \infty$ are lax $\bE_{\bk}$-monoidal functors $\ZZ \to \Sp$ \cite[Theorem 2.2.6.2.]{lurie.higheralgebra}.
Since $(\ZZ,+)$ is initial among monoidal $\infty$-categories consisting of tensor-invertible objects \cite[Theorem 5.1.0.8]{lurie.higheralgebra}, there is a unique monoidal functor $A \smash (S^{-n,-1})^{\smash(-)}: \ZZ \to \Mod_A(\SH(\rS))^{\mathrm{cell}}$ sending $\ell$ to $A \smash S^{-n \ell,-\ell}$.
Composing the latter with the lax symmetric monoidal functor 
$\Gamma: \Mod_A(\SH(\rS))^{\mathrm{cell}}\to \Sp$ gives a lax monoidal functor $\ZZ \to \Sp$ corresponding to an $\bE_1$-algebra $\Theta$ in $\Sp^\ZZ.$

\begin{theorem}\label{0}(Theorem \ref{mot})
Let $\rS$ be a base scheme, $n \in \ZZ$ and $A$ a motivic $\bE_\infty$-ring spectrum over $\rS$. There is a canonical equivalence 
\begin{equation}\label{Eq1}
\Mod_A(\SH(\rS))^{\mathrm{cell}} \simeq \Mod_{\Theta}(\Sp^\ZZ).
\end{equation}
	
\end{theorem}

For any $\ell \in \ZZ$ the shift $(-)\smash S^{n\ell,\ell} $ in $\Mod_A^\mathrm{cell}$
corresponds under equivalence (\ref{Eq1}) to an autoequivalence of $\Mod_{\Theta}(\Sp^\ZZ)$
that precomposes with the translation $$ \ZZ \to \ZZ, \bk \mapsto k-\ell.$$
This guarantees that for any cellular motivic $A$-module spectrum $M$ modeled by a 
$\ZZ$-graded spectrum $M'$ via equivalence (\ref{Eq1}) and any $k,\ell \in \ZZ$ we have an identity of stable homotopy groups $$ \pi_{k,\ell}(M) \simeq \pi_{k-n \ell}(M'_\ell).$$

Equivalence (\ref{Eq1}) is generally not symmetric monoidal.
To promote it to a symmetric monoidal equivalence,
we introduce the following notion of multiplicative periodization:
for any $\n,\m \in \ZZ$ a multiplicative $(\n,\m)$-periodization of a motivic $\bE_\infty$-ring spectrum $A$ is a graded $\bE_\infty$-algebra structure under $A$ on the periodization $ PA=\bigoplus_{\ell \in \ZZ} A \smash S^{\n \ell,\m\ell}$ whose multiplication for $\ell, \ell' \in \ZZ$ is the canonical equivalence $$ A \smash S^{\n \ell,\m\ell} \smash_A  A \smash S^{\n \ell',\m\ell'} \simeq  A \smash S^{\n (\ell+\ell'),\m(\ell+\ell')}.$$

Stated differently, a multiplicative $(\n,\m)$-periodization of $A$ is a lift of
the unique monoidal functor $\ZZ \to \Mod_A(\SH(\rS))^{\mathrm{cell}}, \ell \mapsto A \smash S^{-n \ell,-\ell}$ to a symmetric monoidal functor. Consequently, a multiplicative $(\n,\m)$-periodization of $A$ gives rise to a lift of the lax monoidal functor $\Gamma \circ A \smash (S^{-n,-1})^{\smash(-)}: \ZZ \to \Sp$ to a lax symmetric monoidal functor and so to a lift of the $\bE_1$-algebra $\Theta$ to an $\bE_{\infty}$-algebra in $\Sp^\ZZ.$
By a result of \cite[Theorem 16.19]{Bachmann} every motivic $\bE_{\infty}$-ring spectrum that admits an $\bE_{\infty}$-orientation also admits a multiplicative periodization. 
Besides the tautological example $MGL$, the motivic spectrum representing algebraic cobordism, the motivic spectra $KGL$ representing Weibels homotopy $K$-theory and $M(\ZZ)$ representing motivic cohomology over a Dedekind domain of mixed characteristic admit an $\bE_{\infty}$-orientation \cite[Proposition 5.10]{Gepner} \cite[Remark 11.2]{Spitzweck2} and so a multiplicative periodization. 

We prove that a multiplictive periodization of a motivic $\bE_\infty$-ring spectrum $A$ gives rise to a lift of equivalence (\ref{Eq1}) to a symmetric monoidal equivalence: 
\begin{theorem}\label{1} (Theorem \ref{mot3}) Let $\rS$ be a base scheme, $n \in \ZZ$ and $A$ a motivic $\bE_\infty$-ring spectrum over $\rS$ that admits a multiplicative $(n,1)$-periodization. 
Then equivalence (\ref{Eq1}) 
$$\Mod_A(\SH(\rS))^{\mathrm{cell}} \simeq \Mod_{\Theta}(\Sp^\ZZ)$$
refines to a symmetric monoidal equivalence.
	
\end{theorem}


Moreover Theorem \ref{1} gives a topological model for complex cellular motivic $A$-modules over any complex motivic $\bE_\infty$-ring spectrum $A$
via a symmetric monoidal equivalence: we prove that for any $m \in \ZZ$ every complex motivic $\bE_\infty$-ring spectrum $A$ admits a multiplicative $(0,m)$-periodization (Proposition \ref{Corf} 1.) to obtain the following corollary:

\begin{corollary}\label{3}(Corollary \ref{motcor}) 
Let $A$ be a complex motivic $\bE_\infty$-ring spectrum.
Then $\Theta$ can be promoted to an $\bE_{\infty}$-algebra and equivalence (\ref{Eq1}) 
\begin{equation}\label{2}
\Mod_A(\SH(\bC))^{\mathrm{cell}} \simeq \Mod_{\Theta}(\Sp^\ZZ)\end{equation} for $n=0$
refines to a symmetric monoidal equivalence.		

\end{corollary}

In general motivic $\bE_\infty$-ring spectra do not admit a multiplicative periodization.
To capture the symmetric monoidal structure in the general situation, we offer another topological presentation, where $\ZZ$ is replaced by the category $\mJ$ 
whose objects are pairs of natural numbers $(n, m)$ and whose morphisms $(n,m)\to (n', m')$ between two pairs of numbers are triples consisting of two injections $\alpha: \underline{n} \to \underline{n'}, \beta: \underline{m} \to \underline{m'}$ and an isomorpism
$\underline{n'} \setminus \alpha(\underline{n}) \cong \underline{m'} \setminus \beta(\underline{m})$ of the complements.
The category $\mJ$ was introduced by Sagave-Schlichtkrull \cite[Definition 4.2]{Sagave2011DiagramSA} in the study of symmetric ring spectra and identified as a model for Quillen's localization of the groupoid of finite sets and bijections \cite[Proposition 4.4]{Sagave2011DiagramSA}. The classifying space of the latter is the infinite loop space of the sphere spectrum \cite[Corollary 4.5]{Sagave2011DiagramSA},
where the infinite loop space structure comes from the symmetric monoidal structure on $\mJ$ given by disjoint union.

Since the infinite loop space of the sphere spectrum is the free grouplike $\bE_\infty$-space generated by a point, symmetric monoidal functors out of
$\mJ$ 
inverting all morphisms classify tensor invertible objects.
So every tensor invertible motivic $A$-module $X$ is the image of $(1,0)$ under a unique symmetric monoidal functor $X^{\ot(-)-(-)}:\mJ \to \Mod_A(\SH(\rS)), (n,m) \mapsto X^{\ot n-m}$ inverting all morphisms.
In particular, for every $n \in \ZZ$ there is a symmetric monoidal functor 
\begin{equation}\label{Kuh}
A \smash (S^{n,1})^{\smash(-)-(-)}: \mJ \to \Mod_A(\SH(\rS))^{\mathrm{cell}} \end{equation} that lifts the monoidal functor
$A \smash (S^{n,1})^{\smash(-)}: \ZZ \to \Mod_A(\SH(\rS))^{\mathrm{cell}} $ 
along the unique monoidal functor $\ZZ \to \mJ^{-1}\mJ, \ell \mapsto (\ell,0).$
By Theorem \ref{0} cellular motivic $A$-modules are modeled by modules over an $\bE_1$-algebra $\Theta$ in $\ZZ$-graded spectra whose $\ell$-th term for $\ell \in \ZZ$ is the spectrum of global sections of the motivic spectrum $(A \smash S^{n,1})^{\smash -\ell} $.
We use the symmetric monoidal functor (\ref{Kuh}) to lift the $\ZZ$-graded $\bE_1$-algebra $\Theta$ in 
motivic $A$-modules to a $\mJ$-parametrized $\bE_\infty$-algebra $\Theta'$
whose value at $(\ell,k) \in \mJ$ is the spectrum of global sections of the motivic spectrum $(A \smash S^{n,1})^{\smash k-\ell}  $.
We prove the following symmetric monoidal version of Theorem \ref{0}:

\begin{theorem}\label{4}(Theorem \ref{mot2})
Let $\rS$ be a base scheme, $n \in \ZZ$ and $A$ a motivic $\bE_\infty$-ring spectrum over $\rS$.
There is a symmetric monoidal equivalence 
$$\Mod_A(\SH(S))^{\mathrm{cell}} \simeq \Mod_{\Theta'}(\Fun^\simeq(\mJ,\Sp))$$
lifting equivalence (\ref{Eq1}) along restriction along the monoidal functor $\ZZ \to \mJ^{-1}\mJ,$
where $	\Fun^\simeq(\mJ,\Sp)$ is the $\infty$-category of functors inverting all morphisms.
	
\end{theorem}

To prove Theorems \ref{0}, \ref{1}, \ref{2}, \ref{4} we develop a machinery to present cellular objects in stable presentably symmetric monoidal $\infty$-categories by modules in graded spectra.
Given a motivic $\bE_{\infty}$-ring spectrum $A$ a motivic $A$-module
is cellular if it is generated under arbitrary shifts and small colimits by all (possibly negative) tensor powers of $A \smash S^{2,1}$.
In analogy given a stable presentably symmetric monoidal $\infty$-category $\mD$ containing a tensor invertible object $X$ we call an object of $\mD$ a $X$-cellular object if it belongs to the full subcategory of $\mD$ generated under arbitrary shifts and small colimits by all (possibly negative) tensor powers of $X$
(Definition \ref{cell}).


In analogy to the case of motivic spectra the full subcategory $\Cell_X(\mD)$ 
of $X$-cellular objects in $\mD$ is controlled by bigraded $X$-cellular stable homotopy groups
\begin{equation}\label{eqa}
\pi^X_{p,q}(\Z):= \pi_p(\mD(X^{\ot q},\Z)) \end{equation}
for $\Z \in \Cell_X(\mD)$
and $p,q \in \ZZ.$ For $\mD= \Mod_A(\SH(S))$ for a motivic $E_\infty$-algebra $A$ and $X=A \smash S^{0,1} $ this recovers the stable bigraded homotopy groups.

Equation (\ref{eqa}) suggests that a $X$-cellular object $Z$ is determined by the
$\ZZ$-graded spectrum $\bigoplus_{q \in \ZZ}\mD(X^{\ot q},Z)$.
We prove the following theorem:

\begin{theorem}\label{tre0}(Theorem \ref{tre}, Corollary \ref{Corok0})
Let $\mD$ be a presentably symmetric monoidal $\infty$-category whose tensor unit is compact, and $X$ a tensor invertible object of $\mD.$
The functor 
\begin{equation*}
\Cell_{X}(\mD) \to \Sp^\ZZ, Z \mapsto \bigoplus_{q \in \ZZ}\mD(X^{\ot q},Z)
\end{equation*}
lifts to an equivalence
\begin{equation}\label{0001}
\Cell_{X}(\mD)\to \Mod_{\bigoplus_{q \in \ZZ}\mD(X^{\ot q},\tu_\mD)}(\Sp^\ZZ).\end{equation} 

If the monoidal functor $X^{\ot(-)}: \ZZ \to \mD$ refines to a symmetric monoidal functor, equivalence (\ref{0001}) refines to a symmetric monoidal equivalence.

\end{theorem}
Theorem \ref{tre0} implies Theorems \ref{0}, \ref{1}. Moreover we prove a symmetric monoidal version of Theorem \ref{tre0} (Corollary \ref{Corok}), which gives Theorem \ref{4}. And we show a variant of Theorem \ref{tre0} for filtered objects, which gives Theorem \ref{2}.
To avoid proving all these versions separately, we build up a framework
to treat all these cases simoultaneously (Theorem \ref{tre}).

\vspace{1mm}

\subsection{Overview}

In the following we give an overview about the structure of this work.
 
In section \ref{cellu} we define $X$-cellular objects in a stable presentably symmetric monoidal $\infty$-category $\mD$ with respect to a tensor invertible object $X \in \mD$ (Definition \ref{cell}) abstracting the notion of cellular motivic spectrum from motivic homotopy theory.
While in our applications we only consider $X$-cellular objects for a tensor invertible object $X\in \mD$, for the theory we develop it is natural to allow the following more general cases: 
\begin{enumerate}
\item We allow $X$ to be a dualizable object.
\item We allow $X$ to be a set of dualizable objects.
\item We allow $\mD$ to be $k$-monoidal for any $1 \leq k \leq \infty.$
\end{enumerate}

In section \ref{Dua} we study the interaction between cellular objects and Day-convolution to describe cellular objects in graded spectra.
The stable $\infty$-category $\Sp$ is generated by the tensor unit, the sphere spectrum, under arbitrary shifts and small colimits, and so generated under arbitrary shifts and small colimits by one dualizable object.
We prove that more generally for any small $k$-monoidal $\infty$-category $\mB$, in which every object is dualizable, the $\infty$-category of $\mB$-graded objects in $\Sp$ - the $\infty$-category of functors $\mB \to \Sp$ endowed with Day-convolution - is generated under arbitrary shifts and small colimits by a set of dualizable objects (Proposition \ref{univcell}).

We combine this result with a universal property of Day-convolution (Proposition \ref{univ}) to extend any $k$-monoidal functor $X: \mB^\op \to \mD$ 
to a left adjoint $k$-monoidal functor $\bar{X}: \Fun(\mB,\Sp) \to \mD$ that lands in the full subcategory $\Cell_X(\mD) \subset \mD$ of objects cellular with respect to the essential image of $X$ (Corollary \ref{tr}). 
In section \ref{Pres} we lift the right adjoint $\gamma: \Cell_X(\mD) \to \Fun(\mB,\Sp)$ of $\bar{X}$ to an equivalence $$\gamma': \Cell_X(\mD) \to \RMod_{\Theta}\Fun(\mB, \Sp)$$ (Theorem \ref{tre}), where $\Theta$ is a $\mB$-graded $\bE_\bk$-algebra in $\Sp.$
We deduce Theorem \ref{tre} from Propositions \ref{univcell} and \ref{tt},
where the latter follows from Lemma \ref{qq} that analyzes the unit of an induced adjunction on $\infty$-categories of modules
when evaluated at free modules on dualizable objects. 
In section \ref{sec4} we apply our theory to cellular stable motivic homotopy theory.

\subsection{Notation and Terminology}\label{No}

\begin{itemize}
\item For any $\infty$-category $\mC$ and objects $\X, \Y$ of $\mC$ let $\mC(\X,\Y) $ be the space of maps from $\X $ to $\Y $ in $\mC $ and $\mC^\simeq $ the maximal subspace in $\mC$.

\item For any $\infty$-categories $\mA,\mB$ let $ \Fun{(\mA,\mB}) $ be the $\infty$-category of functors.

\item Let $ \mS$ be the $\infty$-category of spaces and $\Sp$ the $\infty$-category of spectra.

\item Let $\Cat_{\infty}$ be the $\infty$-category of small $\infty$-categories,
$\widehat{\Cat}_{\infty}$ the $\infty$-category of not necessarily small $\infty$-categories and $\Cat_{\infty}^{\rc\rc} \subset \widehat{\Cat}_{\infty}$
the subcategory of $\infty$-categories having small colimits and functors preserving small colimits.

\item If $\mD$ is an $\infty$-category with zero object, let $ \Sigma:  \mD \to \mD $ be the suspension and 
$ \Omega:  \mD \to \mD $ the loops if the necessary pushouts/ pullbacks exist.
In this case we obtain an adjunction $\Sigma: \mD \rightleftarrows \mD: \Omega.$ 

\item An object $\X$ of an $\infty$-category $\mC$ is compact if the functor $\mC(\X,-):\mC \to \mS$ preserves filtered colimits. 

\item For every monoidal $\infty$-category $\mC$ we write $\tu_\mC$ for the tensor unit.

\item An object $\X$ of a monoidal $\infty$-category $\mC$ is a tensor inverse of
an object $\Y$ of $\mC$ if there are equivalences $ \Y \ot \X \simeq \tu_\mC, \X \ot \Y \simeq \tu_\mC.$ 
	
	

\end{itemize}

\subsection{Acknowledgements}

We like to thank Peter Arndt, David Gepner, Thomas Nikolaus and Markus Spitzweck for helpful and inspiring discussions.


\vspace{2mm} 

\section{Cellular objects}

In this section we define cellular objects with respect to any dualizable object of a stable presentably symmetric monoidal $\infty$-category (Definition \ref{cell}). We start with introducing the basic notions we are using throughout this article.

In the following we work with $\bE_{\bk}$-monoidal $\infty$-categories \cite[Definition 2.1.2.13]{lurie.higheralgebra}, where $\bE_{\bk}$ is the $\bk$-th little cubes $\infty$-operad \cite[\text{Definition} \ 5.1.0.4]{lurie.higheralgebra} for $0 \leq \bk \leq \infty$.
Given an $\bE_\bk$-monoidal $\infty$-category $\mC$ let $\tu_\mC$ be the 
tensor unit.
We say that an $\bE_\bk$-monoidal $\infty$-category $\mC$ is compatible with small colimits if $\mC$ admits small colimits and the tensor product 
preserves small colimits component-wise.
Associated to the notion of $\bE_\bk$-monoidal $\infty$-category are the notions of
(lax) $\bE_\bk$-monoidal functor \cite[Definition 2.1.2.7, Definition 2.1.3.7]{lurie.higheralgebra}.
For any $\bE_\bk$-monoidal $\infty$-categories $\mA, \mB$ we write
$$\Fun^{\ot,\bE_\bk, \lax}(\mA,\mB)$$ for the $\infty$-category of lax $\bE_\bk$-monoidal functors $\mA \to \mB$ and $$\Fun^{\ot, \bE_\bk}(\mA,\mB) \subset \Fun^{\ot,\bE_\bk, \lax}(\mA,\mB)$$ for the full subcategory of $\bE_\bk$-monoidal functors.
Moreover we write $\Alg_{\bE_{\bk}}(\mB):= \Fun^{\ot,\bE_\bk, \lax}(*,\mB)$ for the $\infty$-category of $\bE_\bk$-algebras in $\mB$, which are by definiton the lax $\bE_\bk$-monoidal functors $* \to \mB.$
By \cite[Proposition 3.2.1.8, Lemma 3.2.1.10]{lurie.higheralgebra} there is an initial $\bE_{\bk}$-algebra in $\mB$ lying over the tensor unit $\tu_\mB.$

By \cite[Theorem 5.1.2.2]{lurie.higheralgebra} for any $0 \leq \bk, \ell \leq \infty$ the Boardmann-Vogt tensor product $\bE_{\bk} \ot \bE_\ell$ is
$\bE_{\bk+\ell}.$
This guarantees that for any $\bk + \ell$-monoidal $\infty$-category $\mB$ the $\infty$-category 
$\Alg_{\bE_{\bk}}(\mB)$ is $\bE_\ell$-monoidal and there is a canonical equivalence $$\Alg_{\bE_\ell}(\Alg_{\bE_{\bk}}(\mB)) \simeq \Alg_{\bE_{\bk+\ell}}(\mB)$$
\cite[Proposition 3.2.4.3]{lurie.higheralgebra}.
Moreover we use $\bE_\bk$-monoidal $\infty$-categories equipped with a left action of an $\bE_{\bk+1}$-monoidal $\infty$-category $\mC.$ Let $\mathbb{LM}$ be the $\infty$-operad governing left modules \cite[Notation 4.2.1.6.]{lurie.higheralgebra}.
An $\bE_\bk$-monoidal $\infty$-category left tensored over $\mC$ is a $\mathbb{LM} \ot \bE_{\bk}$-monoidal $\infty$-category, whose restriction to $\bE_1 \ot \bE_\bk\simeq \bE_{\bk+1}$ is $\mC$. So an $\bE_\bk$-monoidal $\infty$-category left tensored over $\mC$ has an underlying $\bE_\bk$-monoidal $\infty$-category and an underlying $\infty$-category $\mD$ left tensored over $\mC$ that are compatible to each other. In particular, the left action functor $\mC \times \mD \to \mD$ is
an $\bE_\bk$-monoidal functor.
An $\bE_\bk$-monoidal $\infty$-category left tensored over $\mC$ is compatible with smalll colimits if the underlying $\bE_\bk$-monoidal $\infty$-category is compatible with small colimits and the left $\mC$-action preserves small colimits component-wise.
A (lax) $\mC$-linear (lax) $\bE_\bk$-monoidal functor $\mD \to \mE$ between $\bE_\bk$-monoidal $\infty$-categories left tensored over $\mC$ is a (lax) $\mathbb{LM} \ot \bE_{\bk}$-monoidal functor \cite[Definition 2.1.2.7, Definition 2.1.3.7]{lurie.higheralgebra}.

For any two $\bE_\bk$-monoidal $\infty$-categories $\mA, \mB$ left tensored over $\mC$ we write
$$\Fun_\mC^{\ot, \bE_\bk}(\mA,\mB) $$ for the full subcategory of $\bE_\bk$-monoidal $\mC$-linear functors.

\subsection{Dualizable and tensor invertible objects}\label{cellu}

\begin{recollection}\label{com}
	
Let $\mC$ be a monoidal $\infty$-category.
An object $\X $ of $\mC$ is a left dual of an object $\Y$ of $\mC$ 
if there are morphisms $ \eta: \tu_\mC \to \Y \ot \X $ and $ \varepsilon: \X \ot \Y \to  \tu_\mC $ in $\mC$ such that the following compositions are homotopic to the identity:
\begin{equation*}\label{tumba}
\X \simeq \X \ot \tu_\mC \xrightarrow{\X \ot \eta} \X \ot \Y \ot \X \xrightarrow{\varepsilon \ot \X} \tu_\mC \ot \X \simeq \X, $$$$ \Y \simeq \tu_\mC \ot \Y \xrightarrow{\eta \ot \Y} \Y \ot \X \ot \Y \xrightarrow{\Y \ot \varepsilon} \Y \ot \tu_\mC \simeq \Y.\end{equation*}


These identifications give adjunctions $$\X \ot (-): \mC \rightleftarrows \mC : \Y \ot (-), (-) \ot \Y: \mC \rightleftarrows \mC : (-) \ot \X.$$
So if a left (right) dual exists, it is unique, and 
we write $\X^\dual$ for the left dual of $\X$ and $\X^*$ for the right dual of $\X$ if they exist.
An object of $\mC$ is left (right) dualizable if it admits a left (right) dual. An object of $\mC$ is dualizable if it is left and right dualizable.
A monoidal $\infty$-category is left rigid, right rigid, rigid, respectively, if every object is left dualizable, right dualizable, dualizable, respectively.
Every tensor inverse is a left and right dual.

By the latter adjunctions left duals and right duals are compact if the tensor unit is compact and the tensor product 
preserves filtered colimits.



For any $\bE_2$-monoidal $\infty$-category $\mC$
an object $\X$ of $\mC$ is a left dual of an object $\Y$ of $\mC$ if and only if $\X$ is a right dual of $\Y$, and we say that $\X$ is a dual of $\Y$.

\end{recollection}

\begin{remark}\label{spa}Let $\mC$ be a monoidal $\infty$-category containing
right dualizable objects $\A,\B$ and $\phi: \A \to \B$ a morphism. A right dual of $\phi$ in $\Fun([1],\mC) $ is an inverse of $\phi^*: \B^* \to \A^*.$
Thus $\phi$ is right dualizable in $\Fun([1],\mC) $ if and only if $\phi$ (equivalently $\phi^*$) is an equivalence.

This implies that for every $1 \leq \bk \leq \infty$ and $\bE_\bk$-monoidal $\infty$-categories $\mB, \mD$, where $\mB$ is right rigid, the $\infty$-category $\Fun^{\ot,\bE_\bk}(\mB,\mD) $ is a space: a morphism in $\Fun^{\ot,\bE_\bk}(\mB,\mD) $
corresponds to an $\bE_\bk$-monoidal functor $\mB \to \Fun([1],\mD)$
that sends any $\Z \in \mB$ to a right dualizable object of $ \Fun([1],\mD)$ that is an equivalence.

\end{remark}
 
\begin{notation}
Let $\mC$ be a monoidal $\infty$-category.
Let $ \Pic(\mC) \subset \mC^\simeq$ be the full subspace of tensor invertible objects.
\end{notation}

\begin{example}\label{QS}Let $1 \leq \bk \leq \infty.$
The space $ \Omega^\bk(\Sigma^\bk(S^0))$
is the free grouplike $\bE_{\bk}$-space generated by a point \cite[Theorem 5.2.6.10]{lurie.higheralgebra}.
So for every $\bE_\bk$-monoidal $\infty$-category $\mD$ evaluation at $1$ induces an equivalence
$$\Fun^{\ot,\bE_\bk}(\Omega^\bk(\Sigma^\bk(S^0)),\mD) \simeq \Pic(\mD).$$
Note that the left hand $\infty$-category is a priori a space (Corollary \ref{spa}).
So every object $\X$ of $\Pic(\mD) $ is the image of $1 \in \pi_0(\Omega^\bk(\Sigma^\bk(S^0)))=\ZZ $ under a unique monoidal functor $ \X^{\ot (-)} : \Omega^\bk(\Sigma^\bk(S^0)) \to \mD $.

We mainly apply this to $\bk=1$, where $\Omega(S^1)=\ZZ$, and
$\bk=\infty$.

	
\end{example}

We will use the following model for $ \Omega^\infty(\Sigma^\infty(S^0)):$

\begin{notation}\label{jona}	
Let $\mJ$ be the category whose objects are pairs of natural numbers
$(\n, \m)$ and whose morphisms $(\n,\m) \to (\n', \m')$
are triples consisting of two injections $\alpha: \underline{\n} \to \underline{\n'}, \beta: \underline{\m} \to \underline{\m'}$ and an isomorpism
$\underline{\n'} \setminus \alpha(\underline{\n}) \cong \underline{\m'} \setminus \beta(\underline{\m}).$ 
See \cite[Definition 4.2]{Sagave2011DiagramSA} for a proof that $\mJ$ is indeed a category.

Disjoint union makes $\mJ$ to a symmetric monoidal category
\cite[Proposition 4.3]{Sagave2011DiagramSA}.

\end{notation}

Let $\Sigma$ be the groupoid of finite sets and bijections seen as a symmetric monoidal groupoid via disjoint union.
By \cite[Proposition 4.4]{Sagave2011DiagramSA} the symmetric monoidal category $\mJ$ is equivalent to Quillen's localization construction for $\Sigma$,
where Quillen's localization construction produces from a symmetric monoidal groupoid $\mA$ a symmetric monoidal category $\mA^{-1}\mA$ whose classifying space is the universal grouplike $\bE_\infty$-space associated to $\mA.$
Consequently, the classifying space of the symmetric monoidal category $\mJ$ is the universal grouplike $\bE_\infty$-space associated to $\Sigma$, which is the free grouplike $\bE_\infty$-space generated by a point since $\Sigma$ is the free $\bE_\infty$-space generated by a point \cite[Example 3.1.3.14]{lurie.higheralgebra}.
So by Example \ref{QS} the classifying space of the symmetric monoidal category $\mJ$ is $ \Omega^\infty(\Sigma^\infty(S^0))$ \cite[Corollary 4.5]{Sagave2011DiagramSA}.

In other words, there is a symmetric monoidal functor $\mJ \to  \Omega^\infty(\Sigma^\infty(S^0))$, which sends $(1,0)$ to 1,
such that for any $\infty$-category $\mB$ the induced functor
$$ \Fun(\Omega^\infty(\Sigma^\infty(S^0)), \mB) \to \Fun(\mJ, \mB)$$ is fully faithful and the essential image precisely consists of the functors $\mJ \to \mB$ inverting all morphisms.
This guarantees that for any symmetric monoidal $\infty$-category $\mD$ the induced functor
$$ \Fun^{\ot, \bE_\infty}(\Omega^\infty(\Sigma^\infty(S^0)), \mD) \to \Fun^{\ot, \bE_\infty}(\mJ, \mD)$$ is fully faithful and the essential image precisely consists of the symmetric monoidal functors $\mJ \to \mD$ inverting all morphisms.
So by Example \ref{QS} the functor $\Fun^{\ot, \bE_\infty}(\mJ, \mD) \to \Pic(\mD)$
evaluating at $(1,0) \in \mJ$ restricts to an equivalence on the full subcategory of symmetric monoidal functors $\mJ \to \mD$ inverting all morphisms.




\vspace{2mm}

Now we are ready to define cellular objects.

\begin{definition}\label{cell} Let $1 \leq \bk \leq \infty$ and $ \mD $ a stable $\bE_\bk$-monoidal $\infty$-category compatible with small colimits, $\mB$ a small left rigid $\bE_\bk$-monoidal $\infty$-category and $\X: \mB \to \mD$ an $\bE_\bk$-monoidal functor.
Let $$ \Cell_{\X}(\mD) \subset \mD $$
be the smallest full stable subcategory of $\mD$ closed under small colimits containing the essential image of $ \X$.
\end{definition}
We call $ \Cell_{\X}(\mD) $ the $\infty$-category of $\X$-cellular objects of $\mD$. 

\begin{remark}
The full subcategory $\Cell_{\X}(\mD) \subset \mD$ only depends on the essential image of $\X.$
	
\end{remark}

Since the essential image of $\X$ is an $\bE_\bk$-monoidal full subcategory
of $\mD$ and the tensor product of $\mD$ preserves small colimits component-wise, $ \Cell_{\X}(\mD) \subset \mD$ is an $\bE_\bk$-monoidal full subcategory of $\mD.$
Similarly, if $\mC$ is a stable $\bE_{\bk+1}$-monoidal $\infty$-category compatible with small colimits and $\mD$ is a stable $\bE_\bk$-monoidal $\infty$-category left tensored over $\mC$ compatible with small colimits, 
$ \Cell_{\X}(\mD) \subset \mD $ is an $\bE_\bk$-monoidal full subcategory of $\mD$ closed under the left $\mC$-action if $\mC$ is generated under small colimits and arbitrary shifts by the tensor unit.

\begin{example}\label{tens} Let $1 \leq \bk \leq \infty$, let $ \mD $ be a stable $\bE_\bk$-monoidal $\infty$-category compatible with small colimits, and $\X \in \Pic(\mD)$. 
Let $\X^{\ot(-)}: \Omega^\bk(S^\bk) \to \mD$ be the unique $\bE_\bk$-monoidal functor sending $1 \in \pi_0(\Omega^\bk(S^\bk))=\ZZ $ to $\X$ (Example \ref{QS}).
Then $ \Cell_{\X}(\mD)= \Cell_{\X^{\ot(-)}}(\mD)$ is the 
smallest full stable subcategory of $\mD$ closed under small colimits containing
all objects $\X^{\ot \n} $ for $\n \in \ZZ.$
\end{example}

The following example is our motivating example for Definition \ref{cell}:

\begin{example}\label{cella}
	
Let $\rS$ be a Noetherian separated scheme of finite Krull dimension and $\SH(\rS)$ the symmetric monoidal $\infty$-category of motivic spectra over $\rS.$
Let $\A$ be an $\bE_{\infty}$-algebra in $\SH(\rS)$.
For every $\n \in \ZZ$ the motivic sphere $S^{n,1}= \Sigma^\n(S^{0,1})$ is tensor invertible in $\SH(\rS)$ so that $\A \smash S^{n,1}$ is tensor invertible in $\Mod_\A(\SH(\rS))$.
We define the full subcategory of cellular motivic $\A$-module spectra by
$$\Mod_\A(\SH(\rS))^{\mathrm{cell}}:=\Cell_{\A \smash S^{\n,1}} (\Mod_\A(\SH(\rS))),$$ which does not depend on $\n \in \ZZ$.
\end{example}


\subsection{Graded objects are dualizably generated}\label{Dua}

Let $\mC$ be stable presentably $\bE_\bk$-monoidal $\infty$-category generated by the tensor unit under arbitrary shifts and small colimits, where $1 \leq \bk \leq \infty.$
Then $\mC$ is generated under arbitrary shifts and small colimits by one dualizable object.
We prove in this subsection that more generally the $\bE_\bk$-monoidal
$\infty$-category of functors $\Fun(\mB,\mC)$ for any small left rigid $\bE_\bk$-monoidal $\infty$-category $\mB$ is dualizably generated, i.e. generated under arbitrary shifts and small colimits by a set of dualizable objects (Proposition \ref{univcell}). Here we view $\Fun(\mB,\mC)$ as an $\bE_\bk$-monoidal $\infty$-category via Day-convolution (Proposition \ref{Day1}).
Via Proposition \ref{univ} this will be a key result in order to relate $\X$-cellular objects in any stable $\bE_\bk$-monoidal $\infty$-category $\mD$
left tensored over $\mC$ compatible with small colimits, where $\X: \mB^\op \to \mD$ is an $\bE_\bk$-monoidal functor, to modules in $\Fun(\mB,\mC)$ (Corollary \ref{tr}).

We start with recalling the Day-convolution for $\infty$-categories 
\cite[2.2.6]{lurie.higheralgebra}, \cite{https://doi.org/10.48550/arxiv.2006.08269}, \cite{article}, \cite[Theorem 10.13]{Heine2023AnEB}.

\begin{recollection}\label{Day1}\label{Day2}\label{rem1}

Let $1 \leq \bk \leq \infty$, let $\mB$ be a small $\bE_\bk$-monoidal $\infty$-category and $\mC$ an $\bE_\bk$-monoidal $\infty$-category compatible with small colimits. 
\begin{enumerate}
\item By \cite[Theorem 2.2.6.2]{lurie.higheralgebra} the $\infty$-category $\Fun(\mB, \mC)$ carries an $\bE_\bk$-monoidal structure 
compatible with small colimits such that the evaluation functor $\Fun(\mB, \mC) \times \mB \to \mC$
is $\bE_\bk$-monoidal and induces for every $\bE_\bk$-monoidal $\infty$-category $\mA$ an equivalence
\begin{equation}\label{equs}
\Fun^{\ot, \bE_\bk, \lax}(\mA, \Fun(\mB, \mC)) \simeq \Fun^{\ot, \bE_\bk, \lax}(\mA \times \mB, \mC). \end{equation} 
For $\mA$ the final $\bE_\bk$-monoidal $\infty$-category equivalence (\ref{equs})  identifies $\bE_\bk$-algebras in $\Fun(\mB,\mC)$ with lax $\bE_\bk$-monoidal functors $\mB \to \mC.$

\item The tensor product of two functors $\F, \G:\mB \to \mC$ is the left kan-extension of the functor $ \mB \times \mB \xrightarrow{\F \times \G} \mC \times \mC \xrightarrow{\ot_\mC} \mC$ along the functor $\ot_\mB: \mB \times \mB \to \mB.$

\item Let $\alpha: \mA \to \mB$ be a lax $\bE_\bk$-monoidal functor of small $\bE_\bk$-monoidal $\infty$-categories and $\phi: \mC \to \mD$ a lax $\bE_\bk$-monoidal functor of $\bE_\bk$-monoidal $\infty$-categories compatible with small colimits.
The functor $\phi_*: \Fun(\mB, \mC) \to \Fun(\mB, \mD)$ is lax $\bE_\bk$-monoidal
and is $\bE_\bk$-monoidal if $\phi$ is $\bE_\bk$-monoidal and preserves small colimits. 
The functor $\alpha^*: \Fun(\mB, \mC) \to \Fun(\mA, \mC)$ is lax $\bE_\bk$-monoidal
and admits a left adjoint taking left kan extension along $\alpha,$
which is $\bE_\bk$-monoidal if $\alpha$ is $\bE_\bk$-monoidal.
\item By \cite[Proposition 6.18]{https://doi.org/10.48550/arxiv.2006.08269} the Yoneda-embedding $\y_\mB: \mB \to \Fun(\mB^\op,\mS)$ refines to an $\bE_\bk$-monoidal functor that induces for every locally small $\bE_\bk$-monoidal $\infty$-category $\mD$ compatible with small colimits an equivalence
\begin{equation}\label{ujn}
\Fun^{\ot, \bE_\bk, \L}{(\Fun(\mB^\op, \mS),\mD}) \simeq \Fun^{\ot, \bE_\bk}{(\mB ,\mD}), \end{equation} where the left hand side are left adjoint $\bE_\bk$-monoidal functors.

\end{enumerate}

\end{recollection}




\begin{lemma}\label{lan}
	
Let $\mB$ be a small $\infty$-category, $\mC$ an $\infty$-category that admits small colimits, $\Y \in \mB, \C \in \mC$ and $\rH \in \Fun(\mB,\mC)$.
The map $$\Fun(\mB,\mC)(\C \ot \mB(\Y,-), \rH) \to \mC(\C \ot \mB(\Y,\Y), \rH(\Y)) \to \mC(\C,\rH(\Y))$$ induced by the map $ \C \simeq \C \ot \tu \to \C \ot \mB(\Y,\Y)$ is an equivalence.
\end{lemma}

\begin{proof}
There is a canonical equivalence $$\Fun(\mB,\mC)(\C \ot \mB(\Y,-), \rH) \simeq 
\Fun(\mB,\mS)(\mB(\Y,-), \mC(\C,-) \circ \rH) \simeq \mC(\C,\rH(\Y)).$$	
\end{proof}

\begin{lemma}\label{Gen}
	
Let $\mB$ be a small $\infty$-category and $\mC$ an $\infty$-category that admits small colimits. Then $\Fun(\mB,\mC)$ is generated under small colimits by the objects $\C \ot \mB(\Y,-)$ for $\C \in \mC,\Y \in \mB.$	
	
\end{lemma}

\begin{proof}
We first prove that $\Fun(\mB^\simeq,\mC)$ is generated under small colimits by the functors $\C \ot \mB^\simeq(\Y,-): \mB^\simeq \to \mC$ for $\C \in \mC,\Y \in \mB:$ for any functor $\rH: \mB^\simeq \to \mC$ there is an equivalence 
$\rH \simeq \colim_{\Z \in \mB^\simeq} \rH(\Z)\ot \mB^\simeq(\Z,-)$ represented by the equivalence $$ \Fun(\mB^\simeq,\mC)(\rH, \rH') \simeq \lim_{\Z \in \mB^\simeq}\mC(\rH(\Z), \rH'(\Z)) \simeq \lim_{\Z \in \mB^\simeq}\Fun(\mB^\simeq,\mC)(\rH(\Z)\ot \mB^\simeq(\Z,-), \rH')$$$$ \simeq \Fun(\mB^\simeq,\mC)(\colim_{\Z \in \mB^\simeq} \rH(\Z)\ot \mB^\simeq(\Z,-), \rH')$$
for $\rH' \in \Fun(\mB,\mC),$ where the middle equivalence is by Lemma \ref{lan}.

Therefore the functor $\Fun(\mB,\mC) \to \Fun(\mB^\simeq,\mC)$ restricting along
$\mB^\simeq \subset \mB$ admits a left adjoint that sends the functor
$\mB^\simeq \to \mC, \C \ot \mB^\simeq(\Y,-)$ to the functor $\mB \to \mC, \C \ot \mB(\Y,-)$ using Lemma \ref{lan}.
The functor $\Fun(\mB,\mC) \to \Fun(\mB^\simeq,\mC)$ preserves small colimits as such are computed object-wise and so is monadic by \cite[Theorem 4.7.3.5]{lurie.higheralgebra}.
Consequently, $\Fun(\mB,\mC)$ is generated under small colimits by the essential image of the free functor so that the claim follows. 	
	
\end{proof}

\begin{notation}\label{Nott}Let $1 \leq \bk \leq \infty$, let $\mB$ a small $\bE_\bk$-monoidal $\infty$-category and $\mC$ a presentably $\bE_\bk$-monoidal $\infty$-category.
We write $$\lambda: \mC \simeq \Fun(\ast,\mC) \to \Fun(\mB, \mC)$$ for the $\bE_\bk$-monoidal functor induced by the $\bE_\bk$-monoidal functor
$* \to \mB$ (Remark \ref{rem1}).
We write
$$\mK: \mB^\op \xrightarrow{\y_{\mB^\op}} \Fun(\mB,\mS) \to \Fun(\mB, \mC) $$ for the composition of the  $\bE_\bk$-monoidal Yoneda-embedding and the $\bE_\bk$-monoidal functor induced by the unique left adjoint $\bE_\bk$-monoidal functor $(-) \ot \tu_\mC: \mS \to \mC$ (Remark \ref{rem1}).
\end{notation}

\begin{proposition}\label{univcell}Let $1 \leq \bk \leq \infty$, $\mB$ a small left rigid $\bE_\bk$-monoidal $\infty$-category and $\mC$ a stable $\bE_\bk$-monoidal $\infty$-category compatible with small colimits that is generated under arbitrary shifts and small colimits by the tensor unit. 
Then $$ \Cell_{\mK}(\Fun(\mB, \mC)) = \Fun(\mB, \mC).$$

\end{proposition}

\begin{proof}
By Lemma \ref{Gen} the $\infty$-category $\Fun(\mB,\mC)$ is generated under small colimits by the objects $\C \ot \mB(\Y,-)$ for $\C \in \mC,\Y \in \mB.$	
Since the functor $\mC \to \Fun(\mB,\mC), \C \mapsto \C \ot \mB(\Y,-)$ preserves small colimits and $\mC$ is generated under arbitrary shifts and small colimits by the tensor unit, $\Fun(\mB,\mC)$ is generated under arbitrary shifts and small colimits by the objects $\mK(\Y)= \tu \ot \mB(\Y,-)$ for $\Y \in \mB.$	

\end{proof}

We use the following lemma in the next section in the proof of Theorem \ref{tre}:

\begin{lemma}\label{ident}
Let $1 \leq \bk \leq \infty$, let $\mB$ be a small $\bE_\bk$-monoidal $\infty$-category and $\mC$ an $\bE_\bk$-monoidal $\infty$-category compatible with small colimits.
Let $\Z \in \mB$ be an object that admits a left dual $\Z^\dual$ in $\mB.$
The $\Fun(\mB, \mC)$-linear functor 
$$\mK(\Z^\dual)\ot (-): \Fun(\mB, \mC) \to \Fun(\mB, \mC)$$ is equivalent to the functor
$(\Z \ot (-))^*: \Fun(\mB, \mC) \to \Fun(\mB, \mC)$ precomposing with $\Z \ot (-): \mB \to \mB$.

\end{lemma}

\begin{proof}
Let $\lan_\Z: \mC \to \Fun(\mB,\mC)$ be the left adjoint of evaluation at $\Z \in \mB.$
By Lemma \ref{lan} there is a canonical equivalence $\mK(\Z)=\mB(\Z,-)\ot \tu_\mC \simeq \lan_\Z(\tu_\mC).$

For any functor $\F: \mB \to \mD$ there is a canonical equivalence 
$$\lan_\Z(\tu_\mC) \ot \F = \lan_{\ot_{\mB}}(\ot_\mC \circ (\lan_\Z(\tu_\mC) \times \F)) \simeq $$$$\lan_{\ot_{\mB}} (\lan_{(\tu_\mB \times \Z)}(\ot_\mC \circ (\tu_\mC \times \F)))$$$$ \simeq \lan_{\ot_{\mB} \circ (\tu_\mB \times \Z)}(\ot_\mC \circ (\tu_\mC \times \F))\simeq \lan_{\Z \ot (-)}(\F),$$
which is natural in $\F \in \Fun(\mB, \mD).$
Hence for any $\Z \in \mC$ there is an equivalence $$\mK(\Z) \ot (-) \simeq \lan_{\Z \ot (-)}.$$

So there is an equivalence $\mK(\Z^\dual) \ot (-) \simeq \lan_{\Z^\dual \ot (-)}.$

The adjunction $\Z^\dual \ot (-): \mB \rightleftarrows \mB: \Z \ot (-)$
gives rise to an adjunction $$(\Z \ot (-))^*: \Fun(\mB, \mC) \rightleftarrows\Fun(\mB, \mC): (\Z^\dual \ot (-))^*,$$
which identifies $(\Z \ot (-))^*$ with $\lan_{\Z^\dual \ot (-)}.$	
\end{proof}

Next we will prove another universal property of the Day-convolution (Proposition \ref{univ}), which we combine with Proposition \ref{univcell} to relate $\X$-cellular objects in any presentably $\bE_\bk$-monoidal $\infty$-category $\mD$, where $\X: \mB^\op \to \mD$ is an $\bE_\bk$-monoidal functor, to modules in $\Fun(\mB,\mC)$ (Corollary \ref{tr}).

\vspace{1mm}

\begin{notation}\label{laxhom}Let $1 \leq \bk \leq \infty$, let $\mC$ be a presentably $\bE_{\bk+1}$-monoidal $\infty$-category and $\mD$ an $\bE_\bk$-monoidal
$\infty$-category left tensored over $\mC$ compatible with small colimits.	
The composition $\mC^\op \times \mD^\op \times \mD \xrightarrow{\ot \times \id} \mD^\op \times \mD \xrightarrow{\mD(-,-)} \mS$ of lax $\bE_\bk$-monoidal functors corresponds to a lax $\bE_\bk$-monoidal functor
$ \mD^\op \times \mD \to \Fun(\mC^\op,\mS)$ (Proposition \ref{Day1}),
which takes values in $\mC$, and so induces a lax $\bE_\bk$-monoidal functor
$$[-,-]_\mD: \mD^\op \times \mD \to \mC.$$

\end{notation}

\begin{remark}\label{RemA}
The functor $[-, \tu_\mD]_\mD: \mD^\op \to \mC $ refines to a lax $\bE_\bk$-monoidal functor because the tensor unit $\tu_\mD$ of $\mD$ is an $\bE_\bk$-algebra.
\end{remark}

\begin{proposition}\label{univ}
Let $0 \leq \bk \leq \infty$, let $\mB$ be a small $\bE_\bk$-monoidal $\infty$-category, $\mC$ a presentably $\bE_{\bk+1}$-monoidal $\infty$-category
and $\mD$ an $\bE_\bk$-monoidal $\infty$-category left tensored over $\mC$ compatible with small colimits.
The $\bE_\bk$-monoidal $\infty$-category $\Fun(\mB, \mC)$ carries the structure of an $\bE_\bk$-monoidal $\infty$-category left tensored over $\mC$ compatible with small colimits such that the $\bE_\bk$-monoidal functor $\mK: \mB^\op \to \Fun(\mB, \mC)$ induces an equivalence
\begin{equation}\label{equj}
\Fun_{\mC}^{\ot, \bE_\bk, \L}{(\Fun(\mB, \mC),\mD}) \to \Fun^{\ot, \bE_\bk}{(\mB^\op ,\mD}), \end{equation}  where the left hand side are left adjoint $\bE_\bk$-monoidal $\mC$-linear functors.

\end{proposition}

\begin{proof} \cite[Proposition 2.4.2.5]{lurie.higheralgebra} implies that $\bE_\bk$-monoidal $\infty$-categories and $\bE_\bk$-monoidal functors are classified by $\bE_{\bk}$-algebras and maps of $\bE_{\bk}$-algebras in $\Cat_{\infty}$ with respect to the cartesian structure.
By \cite[Corollary 4.8.1.4, Remark 4.8.1.6]{lurie.higheralgebra} the subcategory $\Cat_{\infty}^{\rc\rc} \subset \widehat{\Cat}_{\infty}$
of $\infty$-categories having small colimits and functors preserving small colimits
inherits a closed symmetric monoidal structure such that the inclusion $ \Cat^{\rc\rc}_\infty \subset \widehat{\Cat}_\infty$ is lax symmetric monoidal.
The resulting inclusion $\Alg_{\bE_{\bk}}(\Cat_{\infty}^{\rc\rc}) \subset\Alg_{\bE_{\bk}}(\widehat{\Cat}_\infty)$ identifies the image as
the subcategory of $\bE_\bk$-monoidal $\infty$-categories compatible with small colimits and $\bE_\bk$-monoidal functors preserving small colimits.

By \cite[Theorem 5.1.3.2]{lurie.higheralgebra} for any $\bE_{\bk+1}$-monoidal $\infty$-category $\mC$ compatible with small colimits identified with an $\bE_{\bk+1}$-algebra in $\Cat^{\rc\rc}_\infty$
the $\infty$-categories $\RMod_\mC(\widehat{\Cat}_{\infty})$
of right modules in $\widehat{\Cat}_{\infty} $ and $\RMod_\mC(\Cat^{\rc\rc}_{\infty})$
of right modules in $\Cat^{\rc\rc}_{\infty} $ carry induced $\bE_\bk$-monoidal structures.
\cite[Proposition 2.4.2.5]{lurie.higheralgebra} implies that $\bE_\bk$-algebras  and maps of $\bE_\bk$-algebras in $\RMod_\mC(\widehat{\Cat}_{\infty})$
are classified by $\bE_\bk$-monoidal $\infty$-categories left tensored over $\mC$
and $\bE_\bk$-monoidal $\mC$-linear functors.
The subcategory $\Alg_{\bE_{\bk}}(\RMod_\mC(\Cat^{\rc\rc}_{\infty})) \subset \Alg_{\bE_{\bk}}(\RMod_\mC(\widehat{\Cat}_{\infty}))$ 
identifies with the subcategory of $\bE_\bk$-monoidal $\infty$-categories left tensored over $\mC$ compatible with small colimits and $\bE_\bk$-monoidal $\mC$-linear functors preserving small colimits.
 
By \cite[Remark 4.8.1.8]{lurie.higheralgebra} the lax symmetric monoidal inclusion $\Cat^{\rc\rc}_\infty \subset \widehat{\Cat}_\infty$ admits a symmetric monoidal left adjoint, which sends a small $\infty$-category $\mA$ to $\Fun(\mA^\op, \mS)$ \cite[Theorem 5.1.5.6]{lurie.HTT}. 
So there is an induced adjunction
\begin{equation}\label{aqq}
\Alg_{\bE_\bk}(\widehat{\Cat}_\infty) \rightleftarrows \Alg_{\bE_\bk}(\Cat^{\rc\rc}_\infty).\end{equation}
By Recollection \ref{Day2} 4. the left adjoint sends a small
$\bE_\bk$-monoidal $\infty$-category $\mA$ to the Day-convolution on $\Fun(\mA^\op, \mS)$.

The free-forgetful adjunction $(-)\ot \mC:\Cat^{\rc\rc}_\infty \rightleftarrows \RMod_\mC(\Cat^{\rc\rc}_\infty) $, where the left adjoint is $\bE_\bk$-monoidal and the right adjoint is lax $\bE_\bk$-monoidal \cite[Theorem 5.1.3.2]{lurie.higheralgebra}, gives rise to an adjunction 
$$(-)\ot \mC: \Alg_{\bE_\bk}(\Cat^{\rc\rc}_\infty)\rightleftarrows \Alg_{\bE_\bk}(\RMod_\mC(\Cat^{\rc\rc}_\infty)).$$
So we obtain an adjunction $$\Alg_{\bE_\bk}(\widehat{\Cat}_\infty) \rightleftarrows \Alg_{\bE_\bk}(\Cat^{\rc\rc}_\infty) \rightleftarrows \Alg_{\bE_\bk}(\RMod_\mC(\Cat^{\rc\rc}_\infty)).$$
Hence for $\mD \in \Alg_{\bE_\bk}(\RMod_\mC(\Cat^{\rc\rc}_\infty))$ there is a canonical equivalence $$ \Fun_{\mC}^{\ot,\bE_\bk, \L}{(\Fun(\mB, \mS) \otimes \mC,\mD}) \simeq \Fun{(\mB^\op ,\mD}).$$ 
It remains to identify the $\bE_\bk$-monoidal $\infty$-categories $\Fun(\mB, \mS) \otimes \mC$ and $\Fun(\mB,\mC).$
Since $\mC$ is a presentably $\bE_\bk$-monoidal $\infty$-category, by \cite[Proposition 7.15]{https://doi.org/10.48550/arxiv.2006.08269}
there is a regular cardinal $\kappa$, a small $\bE_\bk$-monoidal $\infty$-category
$\mC'$ compatible with $\kappa$-small colimits and an $\bE_\bk$-monoidal equivalence $\mC \simeq \Ind_\kappa(\mC'),$
where $\Ind_\kappa(\mC') \subset \Fun(\mC'^\op, \mS)$ is the full reflective subcategory of presheaves preserving $\kappa$-small limits
and the $\bE_\bk$-monoidal structure on $\Ind_\kappa(\mC')$ is an
$\bE_\bk$-monoidal localization of the Day-convolution.
As a consequence of \cite[Proposition 4.8.1.17]{lurie.higheralgebra} the $\bE_\bk$-monoidal localization $\Fun(\mC'^\op, \mS) \to \mC$ induces an 
$\bE_\bk$-monoidal localization $\Fun(\mB, \mS) \otimes \Fun(\mC'^\op, \mS) \to \Fun(\mB, \mS) \otimes\mC $.
Since the left adjoint of adjunction (\ref{aqq}) canonically refines to a symmetric monoidal functor, there is an $\bE_\bk$-monoidal equivalence $$\lambda: \Fun(\mB, \mS) \otimes \Fun(\mC'^\op, \mS) \simeq \Fun(\mB \times \mC'^\op,\mS),$$ under which $\Fun(\mB, \mS) \otimes\mC $
corresponds to the full subcategory $\Fun(\mB \times \mC'^\op,\mS)'$ of 
presheaves preserving $\kappa$-small limits in the second variable.
Consequently, $\Fun(\mB \times \mC'^\op,\mS)'$ is an $\bE_\bk$-monoidal localization of
$\Fun(\mB \times \mC'^\op,\mS)$ and the $\bE_\bk$-monoidal equivalence 
$\lambda$ restricts to an $\bE_\bk$-monoidal equivalence $$\Fun(\mB, \mS) \otimes \mC \simeq \Fun(\mB \times \mC'^\op,\mS)'.$$
By Recollection \ref{Day2} for any $\bE_\bk$-monoidal $\infty$-category $\mA$ there is an equivalence
$$ \Fun^{\ot, \bE_\bk, \lax}(\mA,\Fun(\mB \times \mC'^\op,\mS)) \simeq
\Fun^{\ot, \bE_\bk, \lax}(\mA \times \mB \times \mC'^\op, \mS) $$$$ \simeq 
\Fun^{\ot, \bE_\bk, \lax}(\mA \times \mB,\Fun(\mC'^\op,\mS))$$
that restricts to an equivalence 
$ \Fun^{\ot, \bE_\bk, \lax}(\mA,\Fun(\mB, \mS) \otimes \mC) \simeq $$$ \Fun^{\ot, \bE_\bk, \lax}(\mA,\Fun(\mB \times \mC'^\op,\mS)') \simeq
\Fun^{\ot, \bE_\bk, \lax}(\mA \times \mB,\mC).$$

\end{proof}

\begin{remark}\label{laxhom2}
Let $1 \leq \bk \leq \infty$, let $\mB$ be a small $\bE_\bk$-monoidal $\infty$-category and $\mD$ a locally small $\bE_\bk$-monoidal $\infty$-category compatible with small colimits.

\begin{enumerate}
\item For any $\bE_\bk$-monoidal functor $\psi: \mB \to \mD$ the lax $\bE_\bk$-monoidal right adjoint $ \gamma : \mD \to \Fun(\mB^\op, \mS) $
corresponding to $\psi$ under equivalence (\ref{ujn}) corresponds 
under equivalence (\ref{equs}) to the lax $\bE_\bk$-monoidal functor $$ \mB^\op \times \mD \xrightarrow{\psi^\op \times \mD} \mD^\op \times \mD \xrightarrow{\mD(-,-)} \mS,$$
where $\mD^\op \times \mD \xrightarrow{\mD(-,-)} \mS$ corresponds to the $\bE_\bk$-monoidal Yoneda-embedding. 

\item For $\psi=\y_\mB$ 
we find that $\gamma$ is the identity.
Thus the lax $\bE_\bk$-monoidal evaluation functor $ \mB^\op \times \Fun(\mB^\op, \mS) \to \mS$ factors as lax $\bE_\bk$-monoidal functors
$$ \mB^\op \times \Fun(\mB^\op, \mS) \xrightarrow{\y_\mB^\op \times \id} \Fun(\mB^\op, \mS)^\op \times \Fun(\mB^\op, \mS) \xrightarrow{\Fun(\mB^\op, \mS)(-,-)} \mS.$$

\item Let $\F: \mA \to \mB$ be an $\bE_\bk$-monoidal functor.
Taking $\psi$ to be the composition $\y_\mB \circ \F: \mA \to \mB \to \Fun(\mB^\op, \mS)$ we find that $\F$ extends to a left adjoint $\bE_\bk$-monoidal functor $\bar{\F}: \Fun(\mA^\op,\mS) \to \Fun(\mB^\op,\mS)$ along the Yoneda-embeddings. 
As a consequence of Remark \ref{laxhom2} 2. the lax $\bE_\bk$-monoidal right adjoint $\F^*$ of $\bar{\F}$ corresponds to the lax $\bE_\bk$-monoidal functor 
$$ \mA^\op \times \Fun(\mB^\op,\mS) \xrightarrow{\F^\op \times \id} \mB^\op \times \Fun(\mB^\op,\mS) \xrightarrow{\ev} \mS.$$
Thus the map $$\y_\mA \to \F^* \circ \bar{\F} \circ \y_\mC \simeq \F^* \circ \y_\mD \circ \F$$ in $\Fun^{\ot, \lax, \bE_\bk}(\mA,\Fun(\mA^\op, \mS))$ corresponds to a map $$\theta: \mA(-,-) \to \mB(-,-) \circ (\F^\op \times \F)$$ in $\Fun^{\ot, \lax, \bE_\bk}(\mA^\op \times \mA, \mS).$
\end{enumerate}
\end{remark}

\begin{proposition}\label{univ2}
Let $0 \leq \bk \leq \infty$, let $\mB$ be a small $\bE_\bk$-monoidal $\infty$-category, $\mC$ a presentably $\bE_{\bk+1}$-monoidal $\infty$-category
and $\mD$ an $\bE_\bk$-monoidal $\infty$-category left tensored over $\mC$ compatible with small colimits.
For any $\bE_\bk$-monoidal functor $\psi: \mB^\op \to \mD$ the lax $\bE_\bk$-monoidal right adjoint $ \gamma : \mD \to \Fun(\mB, \mC) $ 
corresponding to $\psi$ under equivalence (\ref{equj}) corresponds to the lax $\bE_\bk$-monoidal functor $$\rho: \mB \times \mD \xrightarrow{ \psi^\op \times \mD} \mD^\op \times \mD \xrightarrow{[-,-]_\mD} \mC. $$ 

\end{proposition}

\begin{proof}
Let $\Psi: \Fun(\mB, \mC) \to \mD $ be the left adjoint 
$\bE_\bk$-monoidal $\mC$-linear functor corresponding to $\psi: \mB^\op \to \mD$
and $\gamma: \mD \to \Fun(\mB, \mC)$ the lax $\bE_\bk$-monoidal right adjoint.
Since $\Psi$ is $\bE_\bk$-monoidally left adjoint to $\gamma$, the lax $\bE_\bk$-monoidal functor $$\Fun(\mB, \mC)^\op \times \mD
\xrightarrow{\id \times \gamma} \Fun(\mB, \mC)^\op \times \Fun(\mB, \mC) \xrightarrow{[-,-]_{\Fun(\mB, \mC)}} \mC.$$
identifies with the lax $\bE_\bk$-monoidal functor $$ \Fun(\mB, \mC)^\op \times \mD
\xrightarrow{\Psi^\op \times \id} \mD^\op \times \mD \xrightarrow{[-,-]_\mD} \mC.$$
Since $\psi \simeq \Psi \circ \mK$ the lax $\bE_\bk$-monoidal functor $\rho$ identifies with $$\mB \times \mD\xrightarrow{\mK^\op \times \id} \Fun(\mB, \mC)^\op \times \mD
\xrightarrow{\Psi^\op \times \id} \mD^\op \times \mD \xrightarrow{[-,-]_\mD} \mC.$$ The latter identifies with the lax $\bE_\bk$-monoidal functor $$\mB \times \mD\xrightarrow{\mK^\op \times \id} \Fun(\mB, \mC)^\op \times \mD
\xrightarrow{\id \times \gamma} \Fun(\mB, \mC)^\op \times \Fun(\mB, \mC) \xrightarrow{[-,-]_{\Fun(\mB, \mC)}} \mC,$$
which is $$\mB \times \mD\xrightarrow{\id \times \gamma} \mB \times \Fun(\mB, \mC)
\xrightarrow{\mK^\op \times \id} \Fun(\mB, \mC)^\op \times \Fun(\mB, \mC) \xrightarrow{[-,-]_{\Fun(\mB, \mC)}} \mC.$$
So it remains to see that the latter identifies with the lax $\bE_\bk$-monoidal functor $\mB \times \mD
\xrightarrow{\id \times \gamma} \mB \times \Fun(\mB, \mC) \to \mC$
corresponding to $\gamma$. For that it is enough to see that the lax $\bE_\bk$-monoidal functor 
$$\xi: \mB \times \Fun(\mB, \mC)
\xrightarrow{\mK^\op \times \id} \Fun(\mB, \mC)^\op \times \Fun(\mB, \mC) \xrightarrow{[-,-]_{\Fun(\mB, \mC)}} \mC \subset \Fun(\mC^\op, \mS)$$ identifies with the lax $\bE_\bk$-monoidal functor $ \mB \times \Fun(\mB, \mC) \xrightarrow{\ev} \mC \subset \Fun(\mC^\op, \mS).$
The latter is adjoint to the lax $\bE_\bk$-monoidal functor
$$\kappa: \mC^\op \times \mB \times \Fun(\mB, \mC)\xrightarrow{\id \times \ev} \mC^\op \times \mC\xrightarrow{\mC(-,-)} \mS.$$
We will prove that $\xi$ also is adjoint to $\kappa.$

The lax $\bE_\bk$-monoidal functor $\xi$ is adjoint to the lax $\bE_\bk$-monoidal functor
$$ \mB \times \mC^\op \times \Fun(\mB, \mC)
\xrightarrow{\y_{\mB^\op}^\op \times \id} \Fun(\mB, \mS)^\op \times \mC^\op \times \Fun(\mB, \mC) \xrightarrow{(\tu_\mC \ot (-))_* \times \id}$$$$\Fun(\mB, \mC)^\op \times \mC^\op \times \Fun(\mB, \mC)\xrightarrow{\id \times (-)^{(-)}} \Fun(\mB, \mC)^\op \times \Fun(\mB, \mC)\xrightarrow{\Fun(\mB, \mC)(-,-)} \mS,$$
which identifies with the lax $\bE_\bk$-monoidal functor
$$\mB \times \mC^\op \times \Fun(\mB, \mC)
\xrightarrow{\y_{\mB^\op}^\op \times \id} \Fun(\mB, \mS)^\op \times \mC^\op \times \Fun(\mB, \mC) \xrightarrow{\id \times (-)^{(-)}}$$$$ \Fun(\mB, \mS)^\op \times \Fun(\mB, \mC)\xrightarrow{\id \times \mC(\tu_\mC,-)_*}\Fun(\mB, \mS)^\op \times \Fun(\mB, \mS)\xrightarrow{\Fun(\mB, \mS)(-,-)} \mS.$$
By Remark \ref{laxhom2} 2. the latter identifies with the lax $\bE_\bk$-monoidal functor
$$\mB \times \mC^\op \times \Fun(\mB, \mC) \xrightarrow{\id \times  (\mC(\tu_\mC,-)_* \circ (-)^{(-)})} \mB \times \Fun(\mB, \mS) \xrightarrow{\ev} \mS,$$
which is $\kappa: \mC^\op \times \mB \times \Fun(\mB, \mC)\xrightarrow{\id \times \ev} \mC^\op \times \mC \xrightarrow{(-)^{(-)}} \mC \xrightarrow{\mC(\tu,-)} \mS.$

\end{proof}

\begin{corollary}\label{tr} 
Let $0 \leq \bk \leq \infty$, let $\mB$ be a small left rigid $\bE_\bk$-monoidal $\infty$-category, $\mC$ a stable presentably $\bE_{\bk+1}$-monoidal $\infty$-category
that is generated under arbitrary shifts and small colimits by the tensor unit,
$\mD$ a stable $\bE_\bk$-monoidal $\infty$-category left tensored over $\mC$ compatible with small colimits and $\X: \mB^\op \to \mD$ an $\bE_\bk$-monoidal functor.
The unique left adjoint $\bE_\bk$-monoidal $\mC$-linear functor $\Fun(\mB,\mC) \to \mD$
extending $\X$ (Proposition \ref{univ} 1.) lands in $\Cell_{\X}(\mD).$
	
\end{corollary}

\begin{proof}
Let $\phi: \Fun(\mB,\mC) \to \mD$ be the unique left adjoint $\bE_\bk$-monoidal
$\mC$-linear functor extending $\X$ of Proposition \ref{univ} 1.
Then $\phi \circ \mK \simeq \X$ so that $\phi$ restricts to a functor
$\Cell_{\mK}(\Fun(\mB, \mC)) \to \Cell_\X(\mD).$
So we conclude via Proposition \ref{univcell} that states $ \Cell_{\mK}(\Fun(\mB, \mC)) = \Fun(\mB, \mC)$.
	
\end{proof}

\section{A model for cellular objects}\label{Pres}

In the following we will prove Theorem \ref{tre} presenting cellular objects as modules. We deduce Theorem \ref{tre} from Propositions \ref{tr} and \ref{tt}.
Proposition \ref{tt} is a consequence of Lemma \ref{qq}, which 
studies the unit of an induced adjunction on $\infty$-categories of modules
when evaluated at free modules on dualizable objects. 

We fix the following notation and terminology of \cite[Definition 4.2.1.13]{lurie.higheralgebra}:

\begin{notation}
Let $\mC$ be a monoidal $\infty$-category and $\A$ an $\bE_1$-algebra in $\mC.$
Let $\RMod_\A(\mC)$ be the $\infty$-category of right $\A$-modules in $\mC$
\cite[Definition 4.2.1.13, 4.2.2.10]{lurie.higheralgebra}.

\end{notation}

We will use the following facts about $\infty$-categories of modules:

\begin{notation}\label{Notil}Let $\mC$ be a monoidal $\infty$-category and $\f: \A \to \B$ a map of $\bE_1$-algebras of $\mC$. 
	
\begin{itemize}
\item There is a left $\mC$-action on $\RMod_\A(\mC)$ and a $\mC$-linear 
forgetful functor $\RMod_\A(\mC) \to \mC$ \cite[Construction 4.8.3.24]{lurie.higheralgebra}, which is an equivalence if $\A$ is the tensor unit \cite[Proposition 4.2.4.9]{lurie.higheralgebra}.

\item There is a $\mC$-linear forgetful functor $\f^*: \RMod_{\B}(\mC) \to \RMod_{\A}(\mC) $
that has a $\mC$-linear left adjoint $$ \B \underset{\A}{\ot } (-) : \RMod_{\A}(\mC) \to  \RMod_{\B}(\mC)$$ if $\mC$ admits geometric realizations 
\cite[Proposition 4.6.2.17]{lurie.higheralgebra}.

\item Every lax monoidal functor $\F: \mC \to \mD$ gives rise to a functor $\RMod_\A(\mC) \to \RMod_{\F(\A)}(\mD).$

\item If $\F$ is monoidal, the functor $\RMod_\A(\mC) \to \RMod_{\F(\A)}(\mD)$ is $\mC$-linear, where we view $\RMod_{\F(\A)}(\mD)$ as left tensored over $\mC$ by pulling back along $\F$
\cite[Construction 4.8.3.24]{lurie.higheralgebra}.

\item If $\F$ is monoidal and admits a right adjoint $\G: \mD \to \mC$,
which inherits a canonical structure of a lax monoidal functor,
the $\mC$-linear functor $\RMod_\A(\mC) \to \RMod_{\F(\A)}(\mD)$
admits a right adjoint, that factors as $\RMod_{\F(\A)}(\mD) \to \RMod_{\G(\F(\A))}(\mC) \to \RMod_\A(\mC).$

\end{itemize}

\end{notation}
\vspace{1mm}

The next Lemma \ref{qq} is a key ingredient in the proof of Proposition \ref{tt}. 
  
\begin{lemma}\label{qq}

Let $\phi: \mC \to \mD $ be a monoidal functor that has a right adjoint  $\gamma: \mD \to \mC$ such that $ \mD$ admits geometric realizations.
Let $\A $ be an  $\bE_1$-algebra of $\mC$ and $\f: \phi(\A) \to \B$ a map of $\bE_1$-algebras of $\mD$ adjoint to a map $\g: \A \to \gamma(\B) $ of $\bE_1$-algebras of $\mC.$
Let $ \Phi :\RMod_{\A}(\mC) \to \RMod_{\phi (\A)}(\mD) $ be the induced functor on left modules and $ \Gamma$ the right adjoint. 
Let $\eta$  be the unit of the induced adjunction 
$$ \RMod_{\A}(\mC) \xrightarrow{\Phi }   \RMod_{\phi (\A)} (\mD)  \xrightarrow{\B \underset{ \phi (\A)}{\ot } (-)} \RMod_{\B} ( \mD). $$

Then $\g: \A \to \gamma(\B) $ is an equivalence if and only if $\eta_\X $ is an equivalence for every free $\A$-module $\X$ on a left dualizable object of $\mC.$ 

\end{lemma}

 \begin{proof}
Let $\Y \in \mC$ be left dualizable and $\Y^\dual \in \mC$ the left dual of Y. 
Since left duals are preserved by monoidal functors, $\phi(\Y^\dual) $ is the left dual of $\phi(\Y) $. 
So we obtain adjunctions $$ \Y^\dual \ot (-): \mC \rightleftarrows \mC : \Y \ot (-), \ \phi(\Y^\dual)\ot (-) : \mD \rightleftarrows \mD : \phi(\Y) \ot(-).$$
The structure of a monoidal functor on $ \phi: \mC \to \mD $ provides an equivalence $$ \alpha: (\phi (\Y^\dual) \ot (-)) \circ \phi  \simeq \phi \circ (\Y^\dual \ot (-)) $$ in $\Fun(\mC, \mD), $ which induces an equivalence $$ \beta:  (\Y \ot (-)) \circ \gamma \simeq \gamma \circ (\phi(\Y) \ot (-)) $$ in $\Fun(\mD, \mC) $ between the right adjoints.
The functor $\beta $ is uniquely characterized by the property that there is a canonical commutative square 
$$
\begin{xy}
 \xymatrix{
\id_\mC  \ar[d] \ar[rr]
&&  \gamma \circ (\phi(\Y) \ot (-)) \circ (\phi (\Y^\dual) \ot (-)) \circ \phi  \ar[d]^\simeq
 \\
(\Y \ot (-)) \circ \gamma \circ \phi \circ (\Y^\dual \ot (-))
\ar[rr]^\simeq && \gamma \circ (\phi(\Y)\ot (-)) \circ \phi \circ (\Y^\dual \ot (-)) 
  }
\end{xy} 
$$ 
in $\Fun(\mC, \mC) $.
The canonical morphism $$ (\Y \ot (-)) \circ \gamma \to (\gamma \phi(\Y) \ot (-)) \circ \gamma  \to \gamma \circ (\phi(\Y) \ot (-)) $$ in $\Fun(\mD, \mC) $  also makes the last square commute and is thus homotopic to $ \beta. $

For any monoidal $\infty$-category $\mE$ and $\bE_1$-algebra $\Z$ in $\mE$ let $\F_\Z : \mE \rightleftarrows \RMod_{\Z}(\mE): \V_\Z$ be the free-forgetful adjunction.
The unit $\eta_{ \A \ot \Y} $ of the adjunction
$$ \RMod_{\A}(\mC) \xrightarrow{\Phi} \RMod_{\phi(\A)}(\mD)\xrightarrow{(-) \underset{\phi(\A)}{\ot}\B} \RMod_{\B}(\mD) $$ at $\Y \ot \A$ is the canonical morphism $ \Y \ot \A \to \Gamma \Phi(\Y \otimes \A) \to \Gamma \f^\ast(\Phi(\Y \otimes \A) \otimes_{{\phi}(\A)} \B). $ 

\vspace{1mm}

Thus the statement of the lemma follows from the commutativity of the following diagram in $\mC $ whose top horizontal morphism is the image of $\eta_{\Y \ot \A} $ in $\mC $ under the forgetful functor and whose bottom horizontal morphism is $\Y \ot \g: $ 
 
\medmuskip=0.0mu
 
\begin{footnotesize}
$$
\begin{xy}
\xymatrix@C-=0.5cm{
\V_\A \F_\A( \Y)
\ar[rr] 
\ar[rd]
\ar[ddddd]_/0em/\simeq
&&\V_\A \Gamma \Phi \F_\A( \Y)
\ar[r] 
\ar[d]^/-.3em/\simeq
&\V_\A \Gamma \f^\ast(\Phi\F_\A(\Y) \underset{ \phi (\A)}{\ot }\B)
\ar[d]^/-.3em/\simeq
\\
&\gamma \phi   \V_{\A} \F_\A(\Y)
\ar[r]^/-.3em/\simeq
\ar[ddd]^/-.3em/\simeq
&\gamma \V_{\phi(\A)} \Phi \F_\A(\Y)
\ar[r] 
\ar[d]^/-.3em/\simeq
&\gamma \V_{\phi(\A)} \f^\ast(\Phi\F_\A(\Y) \underset{\phi(\A)}{\ot}\B)
\ar[d]^/-.3em/\simeq
\\
&&\gamma \V_{\phi(\A)}\F_{\phi(\A)} \phi(\Y)
\ar[r] 
\ar[dd] ^/-.3em/\simeq 
&\gamma \V_{\phi(\A)} \f^\ast( \F_{\phi(\A)}\phi(\Y)\underset{\phi(\A)}{\ot}\B)
\ar[d]^/-.3em/\simeq
\\
&&&\gamma \V_{\B} \F_\B \phi(\Y)
\ar[d]^/-.3em/\simeq
\\
&\gamma  \phi (\Y \ot \A) 
\ar[r]^/-.3em/\simeq
&\gamma(\phi(\Y)  \otimes \phi(\A))
\ar[r]
&\gamma(\phi(\Y)\ot \B)
\\
\Y \otimes \A
\ar[rr]
\ar[ru]
&&\Y \ot \gamma(\phi(\A))
\ar[r]
\ar[u]^{\beta_{\phi(\A)}}_/-.3em/\simeq
&\Y \ot \gamma  (\B)
\ar[u]^{\beta_{\B}}_/-.3em/\simeq }
\end{xy}
$$
\end{footnotesize}
We will verify that the diagram commutes.
The commutativity of the top left square of the diagram says that the unit of the adjunction 
\begin{equation}\label{ad1}
\Phi : \RMod_{\A}(\mC)  \rightleftarrows \RMod_{\phi (\A)} ( \mD) : \Gamma \end{equation} lifts the unit of the adjunction \begin{equation}\label{ad2}\phi: \mC \rightleftarrows \mD : \gamma \end{equation} along the forgetful functors. This holds as the functors
$ \V_\A:\RMod_{\A}(\mC) \to \mC $ and  $ \V_{\phi(\A)}:\RMod_{\phi(\A)}(\mD) \to \mD $ are part of a map of adjunctions from adjunction
(\ref{ad1}) to adjunction (\ref{ad2}) by construction of the induced adjunction on modules. 

The bottom left square of the diagram commutes since the unit $ \id_\mC \to \gamma \circ \phi $ of adjunction (\ref{ad2}) refines to a monoidal natural transformation when $ \gamma $ inherits the structure of a lax monoidal functor from the monoidal left adjoint.
The right middle square (with five knots) commutes since the unit of the adjunction $$(-) \underset{\phi (\A)}{\ot}\B: \RMod_{\phi(\A)} (\mD)\rightleftarrows \RMod_{\B} (\mD): \f^* $$ at $\F_{\phi(\A)}(\phi(\Y))$
lies over the map $ \phi(\Y) \ot \f: \phi(\Y) \ot \phi(\A) \to \phi(\Y)\ot \B$ in $\mD.$

The middle square (with five knots) commutes because the equivalence
$$ \F_{\phi(\A)} \circ \phi  \simeq \Phi \circ \F_\A $$
of functors $ \mC \to \RMod_{\phi(\A)}(\mD)$ 
- induced by the equivalence $ \V_\A \circ \Gamma \simeq \gamma \circ \V_{\phi(\A)} $ between right adjoints - lifts the canonical equivalence of functors $ \mC \to \mD:$
$$ ((-) \ot \phi(\A)) \circ \phi \to \phi \circ ((-) \ot \A).$$
\end{proof}

\begin{proposition}\label{tt}
Let $\mC,\mD$ be stable monoidal $\infty$-categories compatible with small colimits and $\phi: \mC \to \mD $ be a monoidal functor that admits a right 
adjoint $\gamma: \mD \to \mC$.
Let $\A $ be an $\bE_1$-algebra of $\mC$ and $\f: \phi(\A) \to \B$ a map of $\bE_1$-algebras of $\mD$, whose adjoint map $\g: \A \to \gamma(\B) $ is an equivalence. 
Let $ \Phi : \RMod_{\A}(\mC) \to \RMod_{\phi (\A)}(\mD) $ be the induced functor on left modules. Let $\mC$ be generated under small colimits and arbitrary shifts by a collection $\mE $ of left dualizable compact objects such that $\phi$ sends every object of $\mE $ to a compact object. 
The functor $$(\B \underset{ \phi (\A)}{\ot } -) \circ \Phi:\RMod_{\A}(\mC) \to   \RMod_{\phi (\A)}(\mD) \to \RMod_{\B} (\mD) $$
is fully faithful and the essential image is generated under small colimits and arbitrary shifts by the free $\B$-modules on the images under $\phi$
of the objects of $\mE$. 

\end{proposition}

\begin{proof}

Because $\mC$ is generated under small colimits and arbitrary shifts by $\mE$ 
and $ \RMod_{\A}(\mC) $ is generated under small colimits by free $\A$-modules \cite[Example 4.7.2.7]{lurie.higheralgebra}, we find that
$ \RMod_{\A}(\mC) $ is generated under small colimits and arbitrary shifts 
by all free $\A$-modules on objects of $\mE.$

The free $\B$-module functor preserves compact objects because its right adjoint preserves filtered colimits. Therefore since $\phi: \mC \to \mD $ sends objects of $\mE $ to compact objects, the induced functor $ (\B \underset{ \phi (\A)}{\ot } (-)) \circ \Phi:  \RMod_{\A}(\mC) \to  \RMod_{\phi (\A)} (\mD) \to \RMod_{\B} (\mD) $ sends free $\A$-modules on objects of $\mE $ to compact objects.
This implies that the right adjoint $\rho$ of the functor $ (\B \underset{ \phi (\A)}{\ot } (-)) \circ \Phi$ preserves small filtered colimits and thus by stability all small colimits. 
Let $\eta $ be the unit of the adjunction 
$ (\B \underset{ \phi (\A)}{\ot } -) \circ \Phi: \RMod_{\A}(\mC) \rightleftarrows \RMod_{\B} (\mD):\rho. $ Since $\g: \A \to \gamma(\B) $ is an equivalence and $\mE$
consists of dualizable objects, Lemma \ref{qq} guarantees that $\eta_{\Sigma^{\n} \X } $ is an equivalence for every free $\A$-module $\X$ on an object of $\mE $ and $\n \in \ZZ.$ So the statement follows since $ \RMod_{\A}(\mC) $ is
generated under small colimits and arbitrary shifts by all free $\A$-modules on objects of $\mE$ and $\rho$ preserves small colimits.

\end{proof}

\vspace{1mm}

Next we deduce Theorem \ref{tre} from Propositions \ref{tr} and \ref{tt}.
For that we need an enhancement of Notation \ref{Notil}
that we prove in Lemma \ref{Enri}:

\begin{lemma}\label{Notil2}Let $1 \leq \bk \leq \infty$, let $\mC$ be a presentably $\bE_{\bk+1}$-monoidal $\infty$-category, $\mD, \mE$ be $\bE_\bk$-monoidal
$\infty$-categories left tensored over $\mC$ compatible with small colimits, $\F: \mD \to \mE$ an $\bE_\bk$-monoidal $\mC$-linear functor preserving small colimits, and $\A \to \B$ a map of $\bE_{\bk}$-algebras in $\mD$. Then

\begin{enumerate}
\item $\RMod_\A(\mD)$ refines to an $\bE_{\bk-1} $-monoidal $\infty$-category left tensored over $\mC$ compatible with small colimits.
	
\item $\RMod_\A(\mD) \to \RMod_{\F(\A)}(\mE)$ refines to an $\bE_{\bk-1}$-monoidal $\mC$-linear functor.

\item $\B \underset{\A}{\ot } (-) : \RMod_{\A}(\mD) \to  \RMod_{\B}(\mD)$ refines to an $\bE_{\bk-1}$-monoidal $\mC$-linear.
	
\end{enumerate} 
	
\end{lemma}

\vspace{1mm}

We say that a stable presentably monoidal $\infty$-category
$\mC$ is compactly generated by the tensor unit if
the tensor unit is compact and generates $\mC$ under arbitrary shifts and small colimits.

\begin{theorem}\label{tre}

Let $1 \leq \bk \leq \infty$, let $\mB$ be a small right rigid $\bE_\bk$-monoidal $\infty$-category, $\mC$ a stable presentably $\bE_{\bk+1}$-monoidal $\infty$-category compactly generated by the tensor unit, and $\mD$ a stable $\bE_\bk$-monoidal
$\infty$-category left tensored over $\mC$ compatible with small colimits whose tensor unit is compact. 
Let $\X: \mB^\op \to \mD $ be an $\bE_\bk$-monoidal functor.
Let $\X': \mB \xrightarrow{\X^\op}\mD^\op \xrightarrow{[-,\tu_\mD]_\mD}\mC$ be the composition of lax $\bE_\bk$-monoidal functors.
Then the right adjoint
\begin{equation*}\label{eq5}
\gamma: \Cell_{\X}(\mD) \to \Fun(\mB, \mC), \Z \mapsto [-,\Z]_\mD \circ \X^\op
\end{equation*}
of the $\bE_\bk$-monoidal $\mC$-linear functor $\phi$ extending $\X$ induces an $\bE_{\bk-1}$-monoidal $\mC$-linear equivalence
\begin{equation*}\label{eq6}
\Cell_{\X}(\mD)\simeq \RMod_{\tu_\mD}(\Cell_{\X}(\mD)) \to \RMod_{\X'}(\Fun(\mB, \mC)).\end{equation*} 

\vspace{1mm}

Moreover there is a canonical commutative square:
$$
\begin{xy}
\xymatrix{
\Cell_{\X}(\mD)  \ar[d]^\simeq \ar[rr]^{\X(\Z) \ot (-)}
&& \Cell_{\X}(\mD) \ar[d]^\simeq
\\
\RMod_{\X'}(\Fun(\mB, \mC)) \ar[rr]^{(\Z^* \ot (-))^*} && \RMod_{\X'}(\Fun(\mB, \mC)).
}
\end{xy} 
$$ 
\end{theorem}

\begin{proof} By Corollary \ref{univcell} the $\bE_\bk$-monoidal functor $\X:\mB^\op \to \mD$ uniquely extends to an $\bE_\bk$-monoidal left adjoint $\mC$-linear functor
$\phi: \Fun(\mB, \mC) \to \Cell_{\X}(\mD)$.
By Corollary \ref{tr} the functor $\phi$ induces an $\bE_\bk$-monoidal left adjoint $\mC$-linear functor $$(\tu_\mD \underset{\phi (\gamma(\tu_\mD))}{\ot }(-)) \circ \Phi: \RMod_{\gamma(\tu_\mD)}(\Fun(\mB, \mC)) \to \RMod_{\tu_\mD }(\Cell_{\X}(\mD)) \simeq \Cell_{\X}(\mD).$$
We apply Proposition \ref{tt} to deduce that the latter functor is an equivalence.
We verify that the assumptions of Proposition \ref{tt} are satisfied.
By Propositon \ref{univcell} the $\infty$-category $\Fun(\mB, \mC) $
is generated under arbitrary shifts and small colimits by the essential image of
the $\bE_\bk$-monoidal functor $\mK: \mB^\op \to \Fun(\mB, \mC)$.
Since $\mB$ is right rigid (so that $\mB^\op$ is left rigid), the essential image of $\mK$ consists of left dualizable objects.
Moreover every left dualizable object of $\Fun(\mB, \mC)$ is compact
since the tensor unit of $\Fun(\mB^\op, \mC) $ is compact (Recollection \ref{com}).
The tensor unit is compact since the tensor unit of $\mC$ is compact and 
the $\bE_\bk$-monoidal functor $\lambda : \mC \to \Fun(\mB^\op, \mC) $ 
of Notation \ref{Nott} left adjoint to evaluation at $\tu_\mB$ preserves compact objects as its right adjoint preserves filtered colimits.
The functor $(\tu_\mD \underset{\phi (\gamma(\tu_\mD))}{\ot }(-)) \circ \Phi$ sends every object in the essential image of $\mK$ to an object in the essential image of $\X: \mB \to \mD$.
Since $\mB$ is right rigid, the essential image of $\X: \mB^\op \to \mD$ consists of left dualizable objects, which are all compact since the tensor unit of $\mD $ is compact (Recollection \ref{com}).
Proposition \ref{Day1} identifies the $\bE_\infty$-algebra $\gamma({\tu_\mD})$
in $\Fun(\mB,\mC)$ with a lax $\bE_\bk$-monoidal functor $\mB \to \mC.$ Proposition \ref{univ} 2. gives a canonical equivalence $\gamma({\tu_\mD}) \simeq \X'$ of lax $\bE_\bk$-monoidal functors $\mB \to \mC$.
This proves the first part.

For any $\Z \in \mB$ let $\Z^*$ be a right dual of $\Z$ in $\mB$.
Then $\mK(\Z^*)$ is the left dual of $\mK(\Z)$ and $\X(\Z^*)$ is the left dual of $\X(\Z)$.
The monoidal functor $\phi: \Fun(\mB, \mC) \to \Cell_{\X}(\mD)$ sends $\mK(\Z^*)$ to $\X(\Z^*)$ and so gives rise to a commutative square
$$
\begin{xy}
\xymatrix{
\Fun(\mB, \mC) \ar[d]^\phi \ar[rr]^{\mK(\Z^*) \ot (-)}
&& \Fun(\mB, \mC) \ar[d]^\phi
\\
\Cell_{\X}(\mD) \ar[rr]^{\X(\Z^*) \ot (-)} && \Cell_{\X}(\mD)
}
\end{xy} 
$$ 
of $\infty$-categories right tensored over $\Fun(\mB, \mC),$
where we view $\Cell_{\X}(\mD)$ as right tensored over $\Fun(\mB, \mC)$ via $\phi.$ 
Passing to right adjoints we obtain the following commutative square of $\infty$-categories right tensored over $\Cell_{\X}(\mD)$, where we view $\Fun(\mB, \mC) $ as right tensored over $\Cell_{\X}(\mD) $ via $\gamma:$ 
$$
\begin{xy}
\xymatrix{
\Cell_{\X}(\mD) \ar[d]^\gamma \ar[rr]^{\X(\Z) \ot (-)}
&& \Cell_{\X}(\mD) \ar[d]^\gamma
\\
\Fun(\mB, \mC) \ar[rr]^{\mK(\Z)\ot (-)} && \Fun(\mB, \mC).
}
\end{xy} 
$$ 

The latter square gives rise to a commutative square
$$
\begin{xy}
\xymatrix{
\Cell_{\X}(\mD)  \ar[d] \ar[rr]^{\X(\Z) \ot (-)}
&& \Cell_{\X}(\mD) \ar[d]^\simeq
\\
\RMod_{\gamma(\tu_\mD)}(\Fun(\mB^\op, \mC)) \ar[rr]^{\mK(\Z)\ot (-)} && \RMod_{\gamma_{(\tu_\mD)}}(\Fun(\mB^\op, \mC)).
}
\end{xy} 
$$ 

Let $\Z^*$ be a right dual of $\Z$ in $\mB$. Then $\Z \simeq (\Z^*)^\dual$
and by Lemma \ref{ident} the $\Fun(\mB, \mC)$-linear functor 
$$\mK(\Z)\ot (-) \simeq \mK((\Z^*)^\dual)\ot (-): \Fun(\mB, \mC) \to \Fun(\mB, \mC)$$ is equivalent to the functor
$(\Z^* \ot (-))^*: \Fun(\mB, \mC) \to \Fun(\mB, \mC)$ precomposing with $\Z^* \ot (-): \mB \to \mB$.

\end{proof}

\begin{remark}\label{rom}
For every $\Z \in \mB$ the free right $\X'$-module on $\mB(\Z,-)\ot \tu_\mC : \mB \to \mC$ corresponds under the equivalence $$\Cell_{\X}(\mD) \simeq \RMod_{\X'}(\Fun(\mB, \mC)) $$ of Theorem \ref{tre} to the object $\X(\Z) \in \mD.$

\end{remark}

\begin{remark}
Let $1 \leq \bk \leq \infty$, let $\tau: \mA \to \mB$ be an $\bE_\bk$-monoidal functor of small left rigid $\bE_\bk$-monoidal $\infty$-categories, $\mC$ a stable presentably $\bE_{\bk+1}$-monoidal $\infty$-category compactly generated by the tensor unit and $\mD$ a stable $\bE_\bk$-monoidal $\infty$-category left tensored over $\mC$ compatible with small colimits whose tensor unit is compact. 
Let $\X: \mB \to \mD $ be an $\bE_\bk$-monoidal functor. Let $\Y:= \X \circ \tau: \mA \to \mD$. Then $\Cell_{\Y}(\mD) \subset \Cell_{\X}(\mD)$ and there is a commutative square of $\bE_\bk$-monoidal $\infty$-categories left tensored over $\mC$ compatible with small colimits:
$$
\begin{xy}
\xymatrix{
\Cell_{\X}(\mD) \ar[rrr]^\simeq
&&& \RMod_{\X'}(\Fun(\mB^\op, \mC)) \ar[d]
\\
\Cell_{\Y}(\mD) \ar[rrr]^\simeq \ar[u]&&& \RMod_{\Y'}(\Fun(\mA^\op, \mC)).
}
\end{xy} 
$$ 

\end{remark}

For the proof of the next lemma we use the following notation:

\begin{notation}Let $2 \leq \bk \leq \infty$ and $\mD$ an $\bE_\bk$-monoidal $\infty$-category. Let $$ \Rig(\mD) \subset \mD$$ be the full subcategory of dualizable objects.	
	
\end{notation}

\begin{lemma}\label{Du}Let $2 \leq \bk \leq \infty$ and $\mD$ a presentably $\bE_\bk$-monoidal $\infty$-category.
The lax $\bE_\bk$-monoidal functor $\hom(-,\tu_\mD): \mD^\op \to \mD$ of Remark \ref{RemA} restricts to an $\bE_\bk$-monoidal functor $\hom(-,\tu_\mD): \Rig(\mD)^\op \to \Rig(\mD)$ inverse to its opposite.
\end{lemma}

\begin{proof}

If $\X \in \Rig(\mD),$ then $\hom(\X,\tu_\mD)$ is the dual of $\X$, whose dual is $\X.$
For $\X, \Y \in \Rig(\mD)$
the structure map $\hom(\X ,\tu_\mD) \ot \hom(\Y,\tu_\mD) \to \hom(\X \ot \Y,\tu_\mD)$
identifies with the equivalence $\X^\dual \ot \Y^\dual \simeq (\X \ot \Y)^\dual.$ 

To complete the proof we show that the $\bE_\bk$-monoidal functor $$\xi: \Rig(\mD) \xrightarrow{\hom(-,\tu_\mD)} \Rig(\mD)^\op \xrightarrow{\hom(-,\tu_\mD)} \Rig(\mD) \subset \Fun(\Rig(\mD)^\op, \mS) $$ 
is adjoint to the lax $\bE_\bk$-monoidal mapping space functor 
of $\Rig(\mD).$
The $\bE_\bk$-monoidal functor $\xi$ is adjoint to the lax $\bE_\bk$-monoidal functor
$$\Rig(\mD)^\op \times \Rig(\mD) \xrightarrow{\id \times \hom(-,\tu_\mD)} \Rig(\mD)^\op \times \Rig(\mD)^\op \xrightarrow{\ot} \Rig(\mD)^\op \xrightarrow{\Rig(\mD)(-, \tu_\mD)} \mS $$
that is equivalent to the lax $\bE_\bk$-monoidal functor
$$\rho: \Rig(\mD)^\op \times \Rig(\mD) \xrightarrow{\hom(-,\tu_\mD) \times \hom(-,\tu_\mD)}\Rig(\mD) \times \Rig(\mD)^\op \xrightarrow{\Rig(\mD)^\op(-,-)} \mS.$$

By Remark \ref{laxhom2} the $\bE_\bk$-monoidal functor
$\hom(-,\tu_\mD): \Rig(\mD) \xrightarrow{} \Rig(\mD)^\op$ induces a lax $\bE_\bk$-monoidal transformation
$\theta: \Rig(\mD)(-,-) \to \rho$ that is an equivalence:
For any $\X, \Y \in \Rig(\mD)$ the component $\theta_{\X, \Y}$ is the map $$\Rig(\mD)(\X,\Y) \to \Rig(\mD)(\hom(\Y,\tu_\mD), \hom(\X ,\tu_\mD)) \simeq$$$$ \Rig(\mD)(\hom(\Y,\tu_\mD) \ot \X, \tu_\mD) \simeq \Rig(\mD)(\X, \hom(\hom(\Y,\tu_\mD),\tu_\mD))$$
induced by the unit $\Y \to \hom(\hom(\Y,\tu_\mD), \tu_\mD)$ that identifies with the equivalence to the double dual.

\end{proof}

\begin{remark}
Let the assumptions of Theorem \ref{tre} be satisfied, $2 \leq \bk \leq \infty$ 
and $\mD$ presentable.
Then the lax $\bE_\bk$-monoidal functor $\X': \mB^\op \xrightarrow{\X^\op}\mD^\op \xrightarrow{[-,\tu_\mD]_\mD}\mC$ 
can be replaced by $\mB \xrightarrow{\X} \mD \xrightarrow{[\tu_\mD,-]_\mD} \mC$
in the statement of the theorem.
\end{remark}

\begin{proof}
Since $\mD$ is presentable and $ \bk \geq 2,$ the $\bE_\bk$-monoidal structure is closed.
Let $\hom(\X,\Y)$ be the internal hom of objects $\X, \Y \in \mD$.
The dual of any dualizable object $\X \in \mD$ is equivalent to $\hom(\X, \tu)$
by its universal property.
The lax $\bE_\bk$-monoidal functor
$[-,\tu_\mD]_\mD: \mD^\op \to \mC$ factors as
$[\tu_\mD,-]_\mD \circ \hom(-,\tu_\mD): \mD^\op \to \mD \to \mC$.
The lax $\bE_\bk$-monoidal functor $\hom(-,\tu_\mD): \mD^\op \to \mD$ of Remark \ref{RemA} restricts to an $\bE_\bk$-monoidal functor $\Rig(\mD)^\op \to \Rig(\mD)$
inverse to its opposite (Lemma \ref{Du}).
Replacing $\X: \mB \to \Rig(\mD) \subset \mD$ by the $\bE_\bk$-monoidal functor
$$\Y: \mB^\op \xrightarrow{\X^\op} \Rig(\mD)^\op \xrightarrow{\hom(-,\tu_\mD)} \Rig(\mD) \subset \mD$$
in Theorem \ref{tre} we find that $\Cell_{\Y}(\mD)= \Cell_{\X}(\mD)$
and $ \Y' \simeq [\tu_\mD,-]_\mD \circ \X$
as lax $\bE_\bk$-monoidal functors.

\end{proof}	

Theorem \ref{tre} implies the following two important corollaries:

\begin{corollary}\label{Corok0}
	
Let $\mC$ be a stable presentably $\bE_2$-monoidal $\infty$-category compactly generated by the tensor unit, $\mD$ a stable monoidal $\infty$-category left tensored over $\mC$ compatible with small colimits whose tensor unit is compact,
and $\X $ a tensor invertible object of $\mD$.
The lax monoidal functor 
\begin{equation*}\label{eq5}
\gamma: \Cell_{\X}(\mD) \to \Fun(\ZZ, \mC), \Z \mapsto [-,\Z]_\mD \circ \X^{\ot (-)}
\end{equation*}
induces a $\mC$-linear equivalence
\begin{equation}\label{eq000}
\Cell_{\X}(\mD) \to \RMod_{[-,\tu_\mD]_\mD \circ \X^{\ot (-)}}(\Fun(\ZZ, \mC)).\end{equation} 

If moreover $\mC$ is symmetric monoidal, $\mD$ is a symmetric monoidal $\infty$-category left tensored over $\mC$ and
the monoidal functor $\X^{\ot(-)}:\ZZ \to \mD$ lifts to a symmetric monoidal
functor, equivalence (\ref{eq000}) refines to a symmetric monoidal equivalence.
	
\end{corollary}

For the next corollary we fix the following notation: for any $\infty$-categories $\mA,\mB$ let $\Fun^\simeq(\mA,\mB) \subset \Fun(\mA,\mB)$ be the full subcategory of functors $\mA \to \mB$ sending all morphisms to equivalences.
By Example \ref{QS} for every symmetric monoidal $\infty$-category 
$\mD$ and tensor invertible object $\X$ of $\mD$
there is a unique symmetric monoidal functor $\X^{\ot (-)-(-)}:\mJ \to \mD$
that sends $(1,0)$ to $\X.$	

\begin{corollary}\label{Corok}
	
Let $\mC$ be a stable presentably symmetric monoidal $\infty$-category compactly generated by the tensor unit, $\mD$ a stable symmetric monoidal
$\infty$-category left tensored over $\mC$ compatible with small colimits whose tensor unit is compact, and $\X $ a tensor invertible object of $\mD$.
The lax symmetric monoidal functor 
\begin{equation*}\label{eq5}
\gamma: \Cell_{\X}(\mD) \to \Fun^\simeq(\mJ^\op, \mC), \Z \mapsto [-,\Z]_\mD \circ \X^{\ot (-)-(-)}
\end{equation*}
induces a symmetric monoidal $\mC$-linear equivalence
\begin{equation*}\label{eq6}
\Cell_{\X}(\mD) \to \RMod_{[-,\tu_\mD]_\mD \circ \X^{\ot (-)-(-)}}(\Fun^\simeq(\mJ^\op, \mC)).\end{equation*} 
	
\end{corollary}

\begin{proof}
	
The symmetric monoidal functor $\mJ^\op \to \Omega^\infty(\Sigma^\infty(S^0))^\op \simeq  \Omega^\infty(\Sigma^\infty(S^0))$
yields a left adjoint symmetric monoidal functor $$\Fun(\mJ^\op, \mC) \to \Fun( \Omega^\infty(\Sigma^\infty(S^0)), \mC)$$ whose lax symmetric monoidal right adjoint is fully faithful and induces an equivalence $\Fun( \Omega^\infty(\Sigma^\infty(S^0)), \mC) \simeq \Fun^\simeq(\mJ^\op, \mC)$
(see Notation \ref{jona} and the following lines).
The lax symmetric monoidal functor $\gamma$ factors as $ \Cell_{\X}(\mD) \to \Fun( \Omega^\infty(\Sigma^\infty(S^0)), \mC), \Z \mapsto [-,\Z]_\mD \circ \X^{\ot (-)-(-)}$ followed by the resulting symmetric monoidal equivalence
$ \Fun( \Omega^\infty(\Sigma^\infty(S^0)), \mC) \simeq \Fun^\simeq(\mJ^\op, \mC).$
Thus the symmetric monoidal $\mC$-linear functor $$\Cell_{\X}(\mD) \to \RMod_{[-,\tu_\mD]_\mD \circ \X^{\ot (-)-(-)}}(\Fun^\simeq(\mJ^\op, \mC))$$ factors as
$$\Cell_{\X}(\mD) \to \RMod_{[-,\tu_\mD]_\mD \circ \X^{\ot (-)-(-)}}(\Fun( \Omega^\infty(\Sigma^\infty(S^0)), \mC)) \simeq $$$$\RMod_{[-,\tu_\mD]_\mD \circ \X^{\ot (-)-(-)}}(\Fun^\simeq(\mJ^\op, \mC)).$$
So we apply Theorem \ref{tre} to see that this functor is an equivalence.
\end{proof} 

\section{A model for cellular motivic spectra}\label{sec4}

In this section we apply Theorem \ref{tre} to motivic homotopy theory.

\vspace{1mm}
We fix the following notation:
\begin{notation}
Let $\rS$ be a Noetherian separated scheme of finite Krull dimension,
which we call a base scheme.
Let $\SH(\rS)$ be the symmetric monoidal $\infty$-category of motivic spectra over $\rS,$
i.e. the $\mathbb{P}^1$-stabilization of the $\bA^1$-localization 
of the $\infty$-category of derived Nisnevich-sheaves on the category $\Sm_\rS$ of smooth schemes over $\rS.$
\end{notation}
We refer to $\bE_\infty$-algebras in $\SH(\rS)$ as motivic $\bE_\infty$-ring
spectra over $\rS.$ 

\begin{notation}Let $\rS$ be a base schme and $\A$ a motivic $\bE_\infty$-ring spectrum over $\rS$. The symmetric monoidal $\infty$-category of cellular motivic
$\A$-module spectra is $$\Mod_\A(\SH(\rS))^{\mathrm{cell}}:=\Cell_{\A \smash S^{0,1}} (\Mod_\A(\SH(\rS)))$$ (Example \ref{cella}).
\end{notation}
Since $\SH(\rS)$ is a stable presentably symmetric monoidal $\infty$-category, there is a unique left adjoint symmetric monoidal functor $\Sp \to \SH(\rS)$ (taking the $\mathbb{P}^1$-stabilization of the $\bA^1$-localization of the constant sheaf) \cite[Corollary 4.8.2.19]{lurie.higheralgebra}, whose lax symmetric monoidal right adjoint taking global sections we call $\Gamma$.


\begin{construction}
Since $S^{\n,1}$ is tensor-invertible in $\SH(\rS)$, by Example \ref{QS} there is a unique monoidal functor $ (S^{\n,1})^{\smash(-)}: \ZZ \to \SH(\rS)$ sending $\ell$ to $S^{n \ell,\ell}$. By Remark \ref{RemA} the functor $\SH(\rS)(-,\A): \SH(\rS)^\op \xrightarrow{} \Mod_\E(\Sp)$ refines to a lax symmetric monoidal functor since $\A$ is an $\bE_\infty$-algebra.
By Recollection \ref{Day2} the composition $$\ZZ=\ZZ^\op \xrightarrow{(S^{\n,1})^{\smash(-)}} \SH(\rS)^\op \xrightarrow{\SH(\rS)(-,\A)} \Mod_\E(\Sp)$$ of lax monoidal functors corresponds to an $\bE_1$-algebra $\Lambda$ in $\Mod_\E(\Sp)^\ZZ.$
\end{construction}

\begin{theorem}\label{mot}Let $\rS$ be a base scheme, $\n \in \ZZ$ and $\A$ a motivic $\bE_\infty$-ring spectrum over $\rS$ and $\E \to \Gamma(\A)$ a map of $\bE_{\infty}$-rings.
There is a $\Mod_\E(\Sp)$-linear equivalence
\begin{equation}\label{kiki}
\chi: \Mod_\A(\SH(\rS))^{\mathrm{cell}} \simeq \Mod_{\Lambda}(\Mod_\E(\Sp)^\ZZ).
\end{equation}

For every $\M \in \Mod_\A(\SH(\rS))^{\mathrm{cell}}$ and $\ell,\m \in \ZZ$ there is a canonical equivalence
$$ \pi_{\ell,\m}(\M) \simeq \pi_{\ell-\n \m}(\chi(\M)_\m).$$
	
\end{theorem}

\begin{proof}
We apply Theorem \ref{tre} to
$\Mod_\A(\SH(\rS))$ and the monoidal functor $\A \smash (S^{\n,1})^{\smash (-)}: \ZZ \to \Mod_\A(\SH(\rS))$ and identify the $\bE_1$-algebra.
There is an equivalence of $\bE_1$-algebras in $\Mod_\E(\Sp)^\ZZ$, i.e. lax monoidal functors $\ZZ \to \Mod_\E(\Sp)$:
$$\Mod_\A(\SH(\rS))(-,\A) \circ (\A \smash -) \circ (S^{\n,1})^{\smash (-)} \simeq \SH(\rS)((S^{\n,1})^{\smash (-)}, \A)=\Lambda.$$
Theorem \ref{tre} implies that $\chi$ is an equivalence and it remains to verify that $ \pi_{\n \m + \ell,\m}(\M) \simeq \pi_{\ell}(\chi(\M)_\m)$ for any $\m,\ell \in \ZZ.$
By Remark \ref{rom} there is a canonical equivalence 
$\chi(\A \smash (S^{\n,1})^{\smash\m}) \simeq \Theta \smash_\E (\E[\m])$ and so for any $\n \in \ZZ$ a canonical equivalence 
$\chi(\A \smash S^{\n\m+\ell,\m}) \simeq  \Theta \smash_\E(\Sigma^\ell(\E)[\m])$.
Hence we obtain an equivalence $$\pi_{\n \m + \ell,\m}(\M) =\pi_0(\Mod_\A(\SH(\rS))^{\mathrm{cell}}(\A \smash S^{\n \m+\ell,\m},\M)) \simeq $$$$\pi_0(\Mod_{\Theta}(\Mod_\E(\Sp)^\ZZ)(\Theta \smash_\E (\Sigma^\ell(\E)[\m]),\chi(\M)))\simeq $$
$$\pi_0((\Mod_\E(\Sp)^\ZZ)(\Sigma^\ell(\E)[\m],\chi(\M)))\simeq
\pi_0((\Mod_\E(\Sp)^\ZZ)(\Sigma^\ell(\E),\chi(\M)[-\m])) \simeq $$
$$\pi_0(\Mod_\E(\Sp)(\Sigma^\ell(\E),\chi(\M)_\m)) \simeq \pi_\ell(\chi(\M)_\m).$$

\end{proof}

\begin{construction}Since $S^{\n,1}$ is tensor-invertible in $\SH(\rS)$, by Example \ref{QS} there is a unique symmetric monoidal functor $ (S^{\n,1})^{\smash(-)-(-)}: \mJ \to \SH(\rS)$ sending $\ell, \ell'$ to $S^{n (\ell-\ell'), \ell-\ell'}$ and sending all morphisms to equivalences.
By Recollection \ref{Day2} the composition  of lax symmetric monoidal functors$$\mJ^\op \xrightarrow{(S^{\n,1})^{\smash(-)-(-)}} \SH(\rS)^\op \xrightarrow{\SH(\rS)(-,\A)} \Mod_\E(\Sp)$$ corresponds to an $\bE_\infty$-algebra $\Xi$ in $\Fun^\simeq(\mJ^\op,\Mod_\E(\Sp)).$
\end{construction}

\begin{theorem}\label{mot2}
Let $\rS$ be a base scheme, $\n \in \ZZ$ and $\A$ a motivic $\bE_\infty$-ring spectrum over $\rS$ and $\E \to \Gamma(\A)$ a map of $\bE_{\infty}$-rings.
There is a symmetric monoidal equivalence 
$$\Mod_\A(\SH(\rS))^{\mathrm{cell}} \simeq \Mod_{\Xi}(\Fun^\simeq(\mJ,\Mod_\E(\Sp)))$$
lifting equivalence (\ref{kiki}) along restriction along the monoidal functor $\ZZ \to \mJ^{-1}\mJ.$
	
\end{theorem}

\begin{proof}

We apply Corollary \ref{Corok} to $\Mod_\A(\SH(\rS))$ and the symmetric monoidal functor $\A \smash (S^{\n,1})^{\smash(-)-(-)}: \mJ \to \Mod_\A(\SH(\rS))$
and identify the $\bE_\infty$-algebra.
There is an equivalence of $\bE_\infty$-algebras in $\Fun^\simeq(\mJ^\op,\Mod_\E(\Sp))$, i.e. lax symmetric monoidal functors $\mJ^\op \to \Mod_\E(\Sp)$:
$$\Mod_\A(\SH(\rS))(-,\A) \circ (\A \smash (-)) \circ (S^{\n,1})^{\smash(-)-(-)} \simeq$$$$ \SH(\rS)((S^{\n,1})^{\smash(-)-(-)}, \A)=\Xi.$$
\end{proof}

For the next application we fix the following terminology, an adaption
of \cite[Definition 4.1]{spitzweck.per}:

\begin{definition}\label{Period}
	
Let $\rS$ be a base scheme, $\n,\m \in \ZZ$ and $\A$ a motivic $\bE_\infty$-ring spectrum over $\rS$.
A multiplicative $(\n,\m)$-periodization of $\A$ is pair $(\X, \alpha)$,
where $\X$ is an $\bE_\infty$-algebra in $\Mod_\A(\SH(\rS))^\ZZ$ and 
$\alpha$ is an equivalence $\A \smash (S^{\n,\m})^{\smash(-)} \simeq \X$
in $\Mod_\A(\SH(\rS))^\ZZ$ such that the following conditions hold:
\begin{enumerate}
\item The morphism $\alpha_0: \A \to \X_0$ underlies the unique map of $\bE_{\infty}$-algebras.
\item For every $\bk, \ell \in \ZZ$ the following square in $\Mod_\A(\SH(\rS)) $ commutes:
$$
\begin{xy}
\xymatrix{
\A \smash (S^{n,m})^{\smash\bk} \smash_\A \A \smash (S^{n,m})^{\smash\ell}\ar[d]^\simeq \ar[rr]
&& \X_\bk \smash_\A \X_\ell \ar[d]
\\
\A \smash (S^{n,m})^{\smash(\bk+\ell)} \ar[rr] && \X_{\bk+\ell}.
}
\end{xy} 
$$ 
\end{enumerate}
\end{definition}

\begin{remark}\label{lifta}
	
Let $(\X, \alpha)$ be a multiplicative $(\n,\m)$-periodization of $\A$.
The $\bE_{\infty}$-algebra $\X$ in $\Mod_\A(\SH(\rS))^\ZZ$ corresponds to a lax symmetric monoidal functor $\X':\ZZ \to \Mod_\A(\SH(\rS)).$
Condition 2. implies that $\X'$ is symmetric monoidal since the right vertical functor in the square is an equivalence.
	
Condition 1. and 2. precisely say that $\alpha:\A \smash (S^{\n,\m})^{\smash(-)} \to \X$ is an equivalence of monoidal functors $\ZZ \to \mathrm{Ho}(\Mod_\A(\SH(\rS)))$ and so is uniquely determined by
its value at 1 by Example \ref{QS}.
So a multiplicative $(\n,\m)$-periodization of $\A$ corresponds to a symmetric monoidal functor $\Y: \ZZ \to \Mod_\A(\SH(\rS))$ and an equivalence $\Y_1 \simeq \A \smash S^{\n,\m} $ in $\Mod_\A(\SH(\rS)) $.

\end{remark}


\begin{remark}\label{ind} 
Let $\rS$ be a base scheme, $\n,\m \in \mathbb{N}$, $\A \to \B$ a map of motivic $\bE_\infty$-ring spectra over $\rS$ and $(\X,\alpha)$ a multiplicative $(\n,\m)$-periodization of $\A$.
Then $(\B \ot_\A \X,\B \ot_\A \alpha)$ is a multiplicative $(\n,\m)$-periodization of $\B.$

\end{remark}



\begin{theorem}\label{mot3}Let $\rS$ be a base scheme, $\n \in \ZZ$, and $\A$ a motivic $\bE_\infty$-ring spectrum over $\rS$ that admits a multiplicative $(\n,1)$-periodization and $\E \to \Gamma(\A)$ a map of $\bE_{\infty}$-rings.
The $\Mod_\E(\Sp)$-linear equivalence $$\Mod_\A(\SH(\rS))^{\mathrm{cell}} \simeq \Mod_{\Lambda}(\Mod_\E(\Sp)^\ZZ)$$ 
of Theorem \ref{mot} refines to a $\Mod_\E(\Sp)$-linear symmetric monoidal equivalence.		
	
\end{theorem}

\begin{proof}
By Remark \ref{lifta} a multiplicative $(\n,1)$-periodization of $\A$
gives a symmetric monoidal functor $\xi: \ZZ \to \Mod_\A(\SH(\rS))$ that sends
1 to $\A \smash S^{\n,1}$.	
By Example \ref{QS} this implies that the underlying monoidal functor of $\xi$
is the monoidal functor $\A \smash (S^{\n,1})^{\smash (-)}: \ZZ \to \Mod_\A(\SH(\rS))$. 
By Recollection \ref{Day2} the composition $$\ZZ=\ZZ^\op \xrightarrow{\xi} \Mod_\A(\SH(\rS))^\op \xrightarrow{\Mod_\A(\SH(\rS))(-,\A)} \Mod_\E(\Sp)$$ of lax symmetric monoidal functors corresponds to an $\bE_\infty$-algebra $\bar{\Lambda}$ in $\Mod_\E(\Sp)^\ZZ$ whose underlying $\bE_1$-algebra is $\Lambda.$
So the result follows from Corollary \ref{Corok0}.

\end{proof}


For the following corollary recall the motivic $\bE_\infty$-ring spectra
$\MGL_\rS$ representing algebraic cobordism, $\KGL_\rS$ representing Weibels homotopy K-theory and $\M\ZZ_\rS$ representing motivic cohomology with $\ZZ$-coefficients, where $\M\ZZ_\rS$ is only defined over a base scheme $S$ over a Dedekind domain of mixed characteristic.

\begin{corollary}\label{eqh}
	
For any base scheme $\rS$ there are symmetric monoidal equivalences
$$\Mod_{\MGL_\rS}(\SH(\rS))^{\mathrm{cell}} \simeq \Mod_{\SH(\rS)((S^{2,1})^{\smash(-)}, \MGL_\rS)}(\Sp^\ZZ). $$ 
$$\Mod_{\KGL_\rS}(\SH(\rS))^{\mathrm{cell}} \simeq \Mod_{\SH(\rS)((S^{2,1})^{\smash(-)}, \KGL_\rS)}(\Sp^\ZZ). $$

If $\rS$ is over a Dedekind domain of mixed characteristic, there is a symmetric monoidal equivalence $$\Mod_{\M\ZZ_\rS}(\SH(\rS))^{\mathrm{cell}} \simeq \Mod_{\SH(\rS)((S^{2,1})^{\smash(-)}, \M\ZZ_\rS)}(Mod_\ZZ(\Sp)^\ZZ).$$ 
	
\end{corollary}

\begin{proof}
We apply Theorem \ref{mot3} to the motivic $\bE_{\infty}$-ring spectra over $\rS:$
$$\A=\MGL_\rS, \KGL_\rS, \M\ZZ_\rS$$
and need to see that these motivic $\bE_{\infty}$-ring spectra admit a multiplicative $(2,1)$-periodization.

In \cite[Theorem 16.19]{Bachmann} a multiplicative $(2,1)$-periodization of $\MGL_\rS$ is constructed.
Recall that an $\bE_\infty$-orientation of a motivic $\bE_\infty$-ring spectrum $\A$ over $\rS$ is a map of motivic $\bE_{\infty}$-ring spectra $\MGL_\rS \to \A.$
By Remark \ref{ind} every $\bE_\infty$-orientation of 
a motivic $\bE_{\infty}$-ring spectrum $\A$ over $\rS$ gives rise to a multiplicative $(2,1)$-periodization of $\A.$ 
The motivic $\bE_{\infty}$-ring spectra $\KGL_\rS, \M\ZZ_\rS$
admit a canonical $\bE_\infty$-orientation \cite[Proposition 5.10]{Gepner},
\cite[Proposition 11.1, Remark 11.2]{Spitzweck2} and so a multiplicative $(2,1)$-periodization.

	
\end{proof}

\begin{remark}
Corollary \ref{eqh} was independently proven by \cite{spitzweck.per}
with different methods using strict models in the setting of model categories.	

\end{remark}

Next we apply Theorem \ref{mot3} to complex cellular motivic spectra.
In order to apply Theorem \ref{mot3} we prove that complex motivic $\bE_\infty$-ring spectra always admit a multiplicative $(0,\m)$-periodization for every $\m \in \ZZ$
(Proposition \ref{Corf}, Corollary \ref{Corper}).
To construct these multiplicative $(0,\m)$-periodizations we lift multiplicative periodizations along the Betti-realization, which is possible
by Lemma \ref{cel} and Proposition \ref{Proq}.
Before we start, we recall the functor of Betti-realization $ \Sm_\bC\to \Sp$ 
that sends a smooth scheme $\Y$ over $\bC$ to $ \Sigma^\infty_+(\Y(\bC))$.
This functor is symmetric monoidal, $\bA^1$-homotopy invariant, satisfies Nisnevich excision and sends the canonical map $\bG_m \to \bA^1$ to a map of spectra whose homotopy fiber $S^1$ is tensor invertible.
This guarantees by the universal property of $\SH(\bC)$
that Betti-realization $ \Sm_\bC\to \Sp$ uniquely extends to a left adjoint symmetric monoidal functor $B: \SH(\bC) \to \Sp$, which we call Betti-realization, too. 





\begin{lemma}\label{cel}

Let $\mB,\mC,\mD$ be symmetric monoidal $\infty$-categories, $\F: \mC \to \mD, \alpha: \mB \to \mD$ symmetric monoidal functors, $\phi: \mB^\simeq \to \mC^\simeq$ a map such that the following triangle commutes:
$$
\begin{xy}
\xymatrix{
\mB^\simeq \ar[rd]^{\alpha^\simeq} \ar[rr]^{\phi}
&& \mC^\simeq \ar[ld]^{\F^\simeq}
\\
& \mD^\simeq.}
\end{xy} 
$$ 
Assume that the following conditions hold:

\begin{enumerate}
\item The morphism spaces of $\mB$ are empty or contractible,
\item For every $\X, \Y\in\mB$ the induced map 
$$\mC(\phi(\X), \phi(\Y)) \to \mD(\F(\phi(\X)),\F(\phi(\Y))) $$
is an equivalence if there is a morphism $\X \to \Y$ in $\mB.$
\item For every $\X_1, ..., \X_\n\in\mB$ for $\n \geq 0$ there is an equivalence $$\lambda:\phi(\X_1) \ot ... \ot \phi(\X_\n) \simeq \phi(\X_1\ot ... \ot\X_\n)$$ in $\mC$ whose image under $\F$ is the canonical equivalence
$$\alpha(\X_1) \ot ... \ot \alpha(\X_\n) \simeq \alpha(\X_1\ot ... \ot\X_\n).$$
\end{enumerate}

Then there is a unique symmetric monoidal functor $\kappa: \mB \to \mC$
inducing $\phi$ on maximal subspaces and lifting $\alpha: \mB \to \mD$ along
$\F:\mC \to \mD,$ whose structure equivalence $\kappa(\X_1) \ot ... \ot \kappa(\X_\n) \simeq \kappa(\X_1\ot ... \ot\X_\n)$ for $\X_1,...,\X_\n \in \mB$ is $\lambda.$

\end{lemma}

\begin{proof}
	
Using $\lambda$ the induced map
$$\mC(\phi(\X_1) \ot ...\ot \phi(\X_\n), \phi(\Y)) \to \mD(\F(\phi(\X_1)) \ot...\ot \F(\phi(\X_\n)), \F(\phi(\Y))) $$
identifies with the induced map
$$\mC(\phi(\X_1 \ot ...\ot \X_\n), \phi(\Y)) \to \mD(\F(\phi(\X_1 \ot...\ot \X_\n)), \F(\phi(\Y)))$$
and so is an equivalence by 2. if there is a map $\X_1 \ot ... \ot \X_\n \to \Y$ in $\mB.$
This implies by Lemma \ref{lemo} that there is a unique lax symmetric monoidal functor $\kappa: \mB \to \mC$ inducing $\phi$ on maximal subspaces and lifting $\alpha: \mB \to \mD$ along $\F:\mC \to \mD.$

We complete the proof by showing that for every $\X_1, ..., \X_\n\in\mB$ for $\n \geq 0$ the
canonical map $\rho: \kappa(\X_1\ot ... \ot\X_\n) \xrightarrow{\lambda^{-1}}\kappa(\X_1) \ot ... \ot \kappa(\X_\n) \to \kappa(\X_1\ot ... \ot\X_\n)$ in $\mC$ is the identity.
Since $\kappa^\simeq \simeq \phi,$ by Condition 2. this is equivalent to show that 
the image under $\F$ of $\rho$ is the identity.
The map $\F(\rho)$ factors as 
$$\F(\kappa(\X_1\ot ... \ot\X_\n)) \xrightarrow{\F(\lambda)^{-1}}\F(\kappa(\X_1)) \ot ... \ot \F(\kappa(\X_\n)) \to \F(\kappa(\X_1\ot ... \ot\X_\n)),$$ which is the identity by 3.
\end{proof}


\begin{proposition}\label{Proq}
For every $\ell, \m, \n\in \mathbb{Z}$ Betti realization $B: \SH(\bC) \to \Sp$ induces an equivalence
$$\SH(\bC)(S^{\n\ell,\m\ell}, S^{\n\ell,\m\ell}) \simeq \Sp(S^{\n\ell},S^{\n\ell}).$$
\end{proposition}

\begin{proof}

The map of spectra $\SH(\bC)(S^{\n\ell,\m\ell}, S^{\n\ell,\m\ell}) \to \Sp(S^{\n\ell},S^{\n\ell})$
induced by Betti realization factors as
$$\SH(\bC)(S^{\n\ell,\m\ell}, S^{\n\ell,\m\ell}) \simeq \SH(\bC) (S^{0,0}, S^{0,0})\to \Sp(S^0,S^0) \simeq \Sp(S^{\n\ell},S^{\n\ell}).$$

Let $\alpha: \Sp \to \SH(\bC) $ be the unique left adjoint symmetric monoidal functor. 
Since $\Sp$ is the initial stable presentably symmetric monoidal $\infty$-category, the composition $B \circ \alpha: \Sp \to \SH(\bC)  \to \Sp$ 
of left adjoint symmetric monoidal functors is the identity.
Hence the induced map of spectra
$$\Sp(S^0,S^0) \to \SH(\bC) (S^{0}, S^{0}) \to \Sp(S^0,S^0)$$
is the identity.
By \cite[Theorem 1]{Levine} the functor $\alpha$ is fully faithful so that the map $\Sp(S^0,S^0) \to \SH(\bC)(S^{0}, S^{0})$ is an equivalence, and the claim follows.
\end{proof}

Definition \ref{Period} has a topological analogue, where we use Remark \ref{lifta}:

\begin{definition} Let $\n \in \ZZ$ and $\E$ an $\bE_\infty$-ring spectrum.
A multiplicative $\n$-periodization of $\E$ is a symmetric monoidal functor
$\X: \ZZ \to \Mod_\E(\Sp)$ equipped with an equivalence $\E \smash S^\n \simeq \X_1.$

\end{definition}

\begin{proposition}\label{Corf}\label{Cort}


For every $\m \in \ZZ$ there is a multiplicative $(0,\m)$-periodization of the motivic sphere spectrum over $\bC$ that is unique with the property that its Betti-realization is the trivial multiplicative $0$-periodization of the sphere spectrum.

\end{proposition}

\begin{proof}
	
By Proposition \ref{Proq} for every $\ell, \m, \n\in \mathbb{Z}$ Betti realization $B: \SH(\bC) \to \Sp$ induces an equivalence
$\SH(\bC)(S^{\n\ell,\m\ell}, S^{\n\ell,\m\ell}) \simeq \Sp(S^{\n\ell},S^{\n\ell}).$
We apply Lemma \ref{cel} to the Betti-realization $B$ and the functor $\ZZ \to \SH(\bC)^\simeq, \ell \mapsto S^{0,\m \ell}$ in order to uniquely lift the constant symmetric monoidal functor $\ZZ \to \Sp$
taking the sphere spectrum to a symmetric monoidal functor $\xi: \ZZ \to \SH(\bC)$ that sends $\ell $ to $S^{0,\m \ell}$ and whose structure morphism
$S^{0,\m \ell} \smash S^{0,\m \ell'} \to S^{0,\m (\ell+\ell')}$
for $\ell, \ell' \in \ZZ$ is the canonical equivalence.
By Remark \ref{lifta} the symmetric monoidal functor $\xi: \ZZ \to \SH(\bC)$
and the identity of $\xi(1)\simeq S^{0,m}$ is a multiplicative $(0,\m)$-periodization of the motivic sphere spectrum over $\bC$.
	
\end{proof}

Remark \ref{ind} gives the following corollary:

\begin{corollary}\label{Corper}
Let $\A$ be a motivic $\bE_{\infty}$-ring spectrum over $\bC.$
For every $\m \in \ZZ$ there is a multiplicative $(0,\m)$-periodization of $\A$:


\end{corollary}

Corollary \ref{Corper} and Theorem \ref{mot3} imply the following corollary:

\begin{corollary}\label{motcor}Let $\A$ be a motivic $\bE_\infty$-ring spectrum over $\bC$ and $\E \to \Gamma(\A)$ a map of $\bE_{\infty}$-rings.
The $\Mod_\E(\Sp)$-linear equivalence 
\begin{equation}\label{eqqp}
\Mod_\A(\SH(\bC))^{\mathrm{cell}} \simeq \Mod_{\SH(\bC)(S^{0,-}, \A)}(\Mod_\E(\Sp)^\ZZ)\end{equation}
of Theorem \ref{mot} refines to a $\Mod_\E(\Sp)$-linear symmetric monoidal equivalence.		

\end{corollary}





\appendix

\section{Appendix}

In this appendix we prove two lemmas, Lemma \ref{lemo} and Lemma \ref{Enri}.
Lemma \ref{lemo} is used to prove Lemma \ref{cel},
Lemma \ref{Enri} is used to prove Lemma \ref{Notil2}.

For the following lemma we use the notion of $\infty$-operad and $\infty$-preoperad
(also known as flagged $\infty$-operad) \cite[Definition 3.2.1, Definition 2.4.3]{CHU2020}.

\begin{lemma}\label{lemo}

Let $\mB,\mC,\mD$ be $\infty$-operads, $\F: \mC \to \mD$ 
a map of $\infty$-operads and $\phi: \mB^\simeq \to \mC^\simeq$ a map
such that the multi-morphism spaces of $\mB$ are empty or contractible
and for every $\X_1,...,\X_\n,\Y \in \mB$ the induced map 
$$\Mul_\mC(\phi(\X_1),...,\phi(\X_\n); \phi(\Y)) \to \Mul_\mD(\F(\phi(\X_1)),...,\F(\phi(\X_\n)); \F(\phi(\Y))) $$
is an equivalence if there is a multi-morphism $\X_1,...,\X_\n \to \Y$ in $\mB$.
The induced map
$$\{\phi\} \times_{\mS(\mB^\simeq,\mC^\simeq)} \Op_\infty(\mB,\mC) \to \{\F^\simeq \circ \phi\} \times_{\mS(\mB^\simeq,\mD^\simeq)} \Op_\infty(\mB,\mD) $$
is an equivalence, where the pullbacks use the functor
$ (-)^\simeq : \Op_\infty \to \mS$ assigning the space of objects.
\end{lemma}

\begin{proof}

By \cite[4.2]{articles} there is an $\infty$-category $\Pre\Op_\infty$ of $\infty$-preoperads
that contains a full subcategory $\Op_\infty$ of $\infty$-operads and comes equipped with
a cartesian fibration $(-)^\simeq: \Pre\Op_\infty \to \mS$ that takes the space of colors. We set $\psi:= \F^\simeq \circ \phi.$
Consequently, the induced map
$$\{\phi\} \times_{\mS(\mB^\simeq,\mC^\simeq)} \Op_\infty(\mB,\mC) \to \{\psi\} \times_{\mS(\mB^\simeq,\mD^\simeq)} \Op_\infty(\mB,\mD) $$
identifies with the map
\begin{equation}\label{uzo}
(\{\mB^\simeq\} \times_\mS \Op_\infty)(\mB,\phi^{*}(\mC)) \to (\{\mB^\simeq\} \times_\mS\Op_\infty)(\mB,\psi^*(\mD))
\end{equation}
induced by the map $\mC \to (\F^\simeq)^*(\mD)$ of $\infty$-preoperads
lying over the identity of $\mC^\simeq.$

By \cite[Corollary 4.2.9]{articles} for any space $\mX$ there is a monoidal structure 
on the $\infty$-category $\prod_{\n\geq0} \Fun(\mX^{\times \n+1},\mS)$,
where for $\mA,\mB \in \prod_{\n\geq0} \Fun(\mX^{\n+1},\mS)$ and $ \X_1,...,\X_\n,\Y \in \mX$ we have:
$$(\mA \ot \mB)(\X_1,...,\X_\n,\Y) \simeq \underset{\bk \geq 0, \Z\in \mX^{\times \bk}, \varphi \in \bk^\n}{\colim} \bigotimes_{1 \leq \bi \leq \bk} \mA((\X_\bj)_{\bj \in \varphi^{-1}(\bi)},\Z_\bi) \ot \mB(\Z_1,...,\Z_\bk,\Y).$$ 
Moreover the $\infty$-category of $\bE_1$-algebras
with respect to this tensor product
$\Alg(\prod_{\n\geq0} \Fun(\mX^{\times \n+1},\mS))$
is the fiber of the cartesian fibration $\Pre\Op_\infty \to \mS$ over the space $\mX$ \cite[Corollary 4.2.9]{articles}.

Let $\Xi \subset \prod_{\n\geq0} \Fun(\mX^{\times \n+1},\mS)$
be the full subcategory spanned by the obects $\mA$
such that for any $\n \geq 0$ and $\X_1,...,\X_\n,\Y \in \mX$
the space $\mA(\X_1,...,\X_\n,\Y)$ is empty whenever
there is no multi-morphism $\X_1,...,\X_\n \to \Y$ in $\mB$.

The full subcategory $\Xi \subset \prod_{\n\geq0} \Fun(\mX^{\times \n+1},\mS)$
is a colocalization, where the colocal equivalences are those maps
$\mA \to \mB$ that induce for any $\n \geq 0$ and $\X_1,...,\X_\n,\Y \in \mX$
an equivalence $$\mA(\X_1,...,\X_\n,\Y) \simeq \mB(\X_1,...,\X_\n,\Y)$$ whenever
there is a multi-morphism $\X_1,...,\X_\n \to \Y$ in $\mB$.
Indeed, for any $\mA \in \prod_{\n\geq0} \Fun(\mX^{\times \n+1},\mS)$
we construct the universal colocal equivalence $\mB \to \mA $
by setting for any $\n \geq 0$ and $\X_1,...,\X_\n,\Y \in \mX$:  $$\mB(\X_1,...,\X_\n,\Y):= \begin{cases}
\mA(\X_1,...,\X_\n,\Y), \ \Mul_\mB(\X_1,...,\X_\n, \Y) \neq \emptyset \\ \emptyset,  \hspace{23,3mm} \Mul_\mB(\X_1,...,\X_\n, \Y) = \emptyset.
\end{cases}$$

Observe that the monoidal structure on $\prod_{\n\geq0} \Fun(\mX^{\times \n+1},\mS)$ restricts to $\Xi.$
Therefore the embedding $\Xi \subset\prod_{\n\geq0} \Fun(\mX^{\times \n+1},\mS)$
is monoidal, which implies that the right adjoint $\R$ of the embedding
is lax monoidal.
Consequently, we obtain an induced colocalization on $\bE_1$-algebras
$$\Alg(\Xi) \subset \Alg(\prod_{\n\geq0} \Fun(\mX^{\times \n+1},\mS)) \simeq \{\mX\} \times_\mS\Pre\Op_\infty$$ 
whose colocal equivalences lie over the colocal equivalences of the colocalization
$\Xi \subset \prod_{\n\geq0} \Fun(\mX^{\times \n+1},\mS).$
Consequently, for $\mX=\mB^\simeq$ the map (\ref{uzo}) is an equivalence
because the map $\phi^{*}(\mC)\to \psi^*(\mD)$
is a colocal equivalence by assumption.

\end{proof}

By \cite[Corollary 4.8.1.4, Remark 4.8.1.6]{lurie.higheralgebra} the subcategory $\Cat_{\infty}^{\rc\rc} \subset \widehat{\Cat}_{\infty}$
of $\infty$-categories having small colimits and functors preserving small colimits
inherits a closed symmetric monoidal structure such that the inclusion $ \Cat^{\rc\rc}_\infty \subset \widehat{\Cat}_\infty$ is lax symmetric monoidal.


\begin{notation}
Let $\rho: \mU \to \Cat_{\infty}^{\rc\rc}$ be the cocartesian fibration
of symmetric monoidal $\infty$-categories classifying the lax symmetric monoidal inclusion $\Cat_{\infty}^{\rc\rc} \subset \widehat{\Cat}_\infty$.
	
\end{notation}

\begin{lemma}\label{Enri}
	
Let $1 \leq \bk \leq \infty$ and $\mC$ a presentably $\bE_{\bk+1}$-monoidal $\infty$-category. 
There is a canonical functor $$\Alg_{\bE_{\bk}}(\mU) \times_{\Alg_{\bE_{\bk}}(\Cat^{\rc\rc}_{\infty})} \Alg_{\bE_{\bk}}(\RMod_\mC(\Cat^{\rc\rc}_{\infty})) 
\to $$$$ \Alg_{\bE_{\bk-1}}(\RMod(\RMod_\mC(\Cat^{\rc\rc}_{\infty})))$$
over $\Alg_{\bE_{\bk}}(\RMod_\mC(\Cat^{\rc\rc}_{\infty}))$
that induces on the fiber over every $\bE_\bk$-monoidal $\infty$-category left tensored over $\mC$ compatible with small colimits a functor
$$\Alg_{\bE_{\bk}}(\mD) \to \Alg_{\bE_{\bk-1}}(\RMod_\mD(\RMod_\mC(\Cat^{\rc\rc}_{\infty})))$$
that sends an $\bE_{\bk}$-algebra $\A$ in $\mD$ to $\RMod_\A(\mD)$
seen as an $\bE_{\bk-1}$-monoidal $\infty$-category left tensored over $\mC$ compatible with small colimits equipped with a compatible left $\mD$-action.

\end{lemma}

\begin{proof}

By \cite[Construction 4.8.3.24]{lurie.higheralgebra} there is a functor $$\xi: \Alg(\mU^{\rc\rc}) \to \RMod(\Cat^{\rc\rc}_{\infty})$$ over $\Alg(\Cat^{\rc\rc}_\infty)$ that sends a pair $(\mC, \A)$ to the pair $(\mC, \RMod_\A(\mC))$
and induces on the fiber over any $\mC \in \Alg(\Cat^{\rc\rc}_\infty)$ a functor
$\Alg(\mC) \to \RMod_\mC(\Cat^{\rc\rc}_\infty).$

The $\infty$-category $\mU$ admits finite products that are preserved by $\rho$, where the final object of $\mU$ lies over the final $\infty$-category and the product of two objects $(\mC,\X), (\mD,\Y)\in \mU$ is $(\mC \times \mD, (\X,\Y)) \in \mU$.
The functor $\xi$ preserves finite products because for any 
$(\mC, \A), (\mC', \A') \in \Alg(\mU^{\rc\rc})$ the canonical functor
$\RMod_{(\A, \A')}(\mC \times \mC') \to \RMod_\A(\mC) \times \RMod_{\A'}(\mC')$ is an equivalence.
Moreover for any $(\mC_1, \A_1),..., (\mC_\n, \A_\n), (\mD,\B) \in \Alg(\mU^{\rc\rc})$ and monoidal functor $\mu: \mC_1 \times ... \times \mC_\n \to \mD$
that preserves small colimits component-wise and morphism of $\bE_1$-algebras
$ \mu(\A_1,...,\A_\n) \to \B$ in $\mD$ the induced functor
$$ \RMod_{\A_1}(\mC_1) \times ... \times \RMod_{\A_\n}(\mC_\n) \simeq \RMod_{\mu(\A_1,...,\A_\n)}(\mC_1 \times ... \times \mC_\n) \to \RMod_\mB(\mD)$$ 
preserves small colimits component-wise since colimits in modules are taken underlying.
This implies that the functor $\xi$ refines to a symmetric monoidal functor
and so induces a symmetric monoidal functor
$$ \xi': \Alg_{\bE_{\bk}}(\mU^{\rc\rc}) \simeq \Alg_{\bE_{\bk-1}}(\Alg(\mU^{\rc\rc}))
\to \Alg_{\bE_{\bk-1}}(\RMod(\Cat^{\rc\rc}_{\infty}))$$
over $\Alg_{\bE_{\bk}}(\Cat_{\infty}^{\rc\rc}).$
The $\bk$-1-monoidal functor $$\theta: \RMod(\RMod_\mC(\Cat^{\rc\rc}_\infty)) \to \Alg(\RMod_\mC(\Cat^{\rc\rc}_\infty)) \times_{\Alg(\Cat^{\rc\rc}_\infty)} \RMod(\Cat^{\rc\rc}_\infty)$$
is a map of cartesian fibrations over $\Alg(\RMod_\mC(\Cat^{\rc\rc}_\infty))$
that induces on the fiber over any monoidal $\infty$-category $\mD$ left tensored over $\mC$ compatible with small colimits the functor 
$$ \RMod_\mD(\RMod_\mC(\Cat^{\rc\rc}_\infty)) \to \RMod_\mD(\Cat^{\rc\rc}_\infty)$$
over $\Cat^{\rc\rc}_\infty$.
The latter functor is an equivalence by \cite[Corollary 4.7.3.16]{lurie.higheralgebra} since source and target are monadic over $\Cat^{\rc\rc}_\infty$ and induces on monads the equivalence
$  \mD \ot_\mC (\mC \ot (-)) \simeq \mD\ot (-).$
Hence the functor
$\Alg_{\bE_{\bk-1}}(\RMod(\RMod_\mC(\Cat^{\rc\rc}_\infty))) \to $$$ \Alg_{\bE_{\bk}}(\RMod_\mC(\Cat^{\rc\rc}_\infty)) \times_{\Alg_{\bE_{\bk}}(\Cat^{\rc\rc}_\infty)} \Alg_{\bE_{\bk-1}}(\RMod(\Cat^{\rc\rc}_\infty))$$
is an equivalence.
Since the latter functor is an equivalence, the pullback of $\xi'$ along the canonical functor $\Alg_{\bE_{\bk}}(\RMod_\mC(\Cat^{\rc\rc}_\infty)) \to \Alg_{\bE_{\bk}}(\Cat^{\rc\rc}_{\infty})$
gives the desired functor.

\end{proof}


\bibliographystyle{plain}

\bibliography{ma}

\end{document}